\documentclass[12pt]{article}
\usepackage{e-jc}

\usepackage{amsmath}
\usepackage{graphicx}

\usepackage{thm-restate}
\usepackage{enumerate}
\usepackage{complexity}
\usepackage{mathtools}
\usepackage{subcaption}
\usepackage[shortcuts]{extdash}

\usepackage[normalem]{ulem}
\usepackage{blkarray}
\usepackage{listings}
\usepackage{xcolor}

\newcommand{\newV}[1]{#1}
\newcommand{\oldV}[1]{\iffalse{#1}\fi}
\newcommand{\nnewV}[1]{{{#1}}}
\newcommand{\noldV}[1]{{\iffalse{#1}\fi}}
\newcommand{\bothV}[2]{\oldV{#1}\newV{#2}}
\newcommand{\nbothV}[2]{\noldV{#1}\nnewV{#2}}

\graphicspath{{figures/}}

\def\NN{\mathbb{N}}

\renewcommand{\qedsymbol}{\openbox}

\dateline{Aug 22, 2020 }{Dec 8, 2021}{Jan 28, 2022}


\Copyright{The authors. Released under the CC BY license (International 4.0).}
\newcommand*\samethanks[1][\value{footnote}]{\footnotemark[#1]}

\author{
Christoph Hertrich\thanks{All three authors were supported by DFG-GRK 2434 Facets of Complexity.} \qquad Felix Schr\"{o}der\samethanks \qquad Raphael Steiner\samethanks
\\
\small Institute of Mathematics \\[-0.8ex]
\small Technische Universit\"{a}t Berlin \\[-0.8ex]
\small Germany, E.U.\\
\small\tt \string{hertrich,fschroed,steiner\string}@math.tu-berlin.de\\
}

\title{Coloring Drawings of Graphs}

\begin{document}
\maketitle
\begin{abstract}
We consider cell colorings of drawings of graphs in the plane. 
Given a multi-graph $G$ together with a drawing $\Gamma(G)$ in the plane with only finitely many crossings, we define a \emph{cell $k$-coloring} of $\Gamma(G)$ to be a coloring of the maximal connected regions of the drawing,
 the \emph{{cells}}, with $k$ colors such that adjacent {cells} have different colors. By the $4$-color theorem, every drawing of a
  bridgeless graph has a {cell }$4$-coloring. A drawing of a graph is {cell }$2$-colorable if and only if the underlying graph is
   Eulerian. We show that every graph without degree 1 vertices admits a {cell }$3$-colorable drawing. This leads to the natural question which \emph{abstract} graphs have the property that each of {their} drawings has a cell $3$-coloring. We say that such a
    graph is \emph{{universally cell }$3$-colorable}. We show that every $4$-edge-connected graph and every graph admitting a \emph{nowhere-zero $3$-flow} is {universally cell} $3$-colorable.
     We also discuss circumstances under which {universal cell} $3$-colorability guarantees the existence of a nowhere-zero $3$-flow. On the negative side, we present an infinite family of {universally cell }$3$-colorable graphs without a nowhere-zero $3$-flow. On the positive side,
      we formulate a conjecture which has a surprising relation to a famous open problem by Tutte known as the \emph{$3$-flow-conjecture}.
       We prove our conjecture for subcubic and for $K_{3,3}$-minor-free graphs.
\end{abstract}
\section{Introduction} 

Graph coloring is one of the earliest and most influential branches of graph theory, whose first occurences date back more than 150 years. Maybe the most celebrated problem in graph theory is the $4$-color-problem, asking whether the bounded regions of every planar map can be colored using four colors such that regions sharing a common border receive different colors. This problem was finally resolved in the positive in 1972 when Appel and Haken~\cite{4CT1,4CT2} presented a computer-assisted proof of their famous \emph{$4$-Color-Theorem}, which formally states that the chromatic number of every planar graph is at most {$4$}.

In this paper we combine the topics of graph coloring and graph drawing by studying colorings of planar maps arising from drawings\footnote{{Drawings of graphs considered in this paper are allowed to be of a very general form, in that we do not require simplicity of the drawings: Edges are allowed to cross each other arbitrarily often (as long as their intersection consists of finitely many points), even if they share a common endpoint. Self-crossings of edges are not excluded either. A precise definition of what we mean by a drawing of a graph is given at the end of the introduction.}} of possibly \emph{non-{planar}} graphs. Formally, a \emph{cell $k$-coloring} of a drawing $\Gamma$ is a proper coloring of the dual graph $G^\top(\Gamma)$ of $\Gamma$, {that is}, a coloring $c\colon\mathcal{F}(\Gamma) \rightarrow \{0,\ldots,k-1\}$ of the {cells} such that for any two cells $f_1, f_2$ which are \emph{adjacent} in $\Gamma$ ({that is}, they share a common segment of an edge), we have $c(f_1) \neq c(f_2)$. See Fig.~\ref{wild_leftside} for an example. 
\begin{figure}[ht]
\centering
\includegraphics[scale=1]{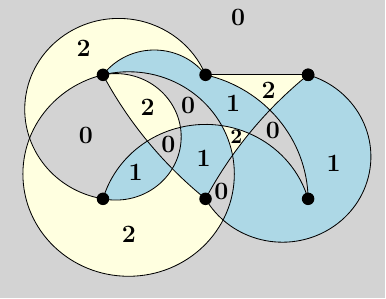}
\caption{A {cell }$3$-colored drawing of a graph.}\label{wild_leftside}
\end{figure}

Note that in a drawing $\Gamma$ it might occur that a {cell} $f$ is adjacent to itself, in which case no {cell }coloring can exist, compare {Fig.}~\ref{selftouch}.

\begin{figure}[ht]
\centering
\includegraphics[scale=0.8]{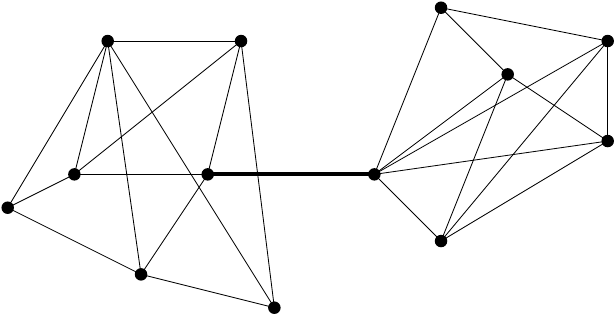}
\caption{A drawing of a graph with a self-touching outer {cell}, caused by the existence of a bridge (edge marked fat) in the underlying graph.}\label{selftouch}
\end{figure}

However, in this case the edge involved in the self-touching must be a bridge of the underlying abstract graph, hence, self-touchings do not occur in drawings of bridgeless graphs. This justifies that most of our results are formulated only for the setting of bridgeless graphs. 
Using the 4-Color-Theorem, we directly see that four colors are sufficient to {cell }color bridgeless graphs.

\begin{proposition}\label{obvious}
Every drawing of a bridgeless graph admits a {cell }$4$-coloring. 
\end{proposition}

\begin{proof}
Let $\Gamma$ be a drawing of {a} bridgeless graph $G$. Let $G_\text{isc}(\Gamma)$ be the planar graph obtained from $G$ by introducing new vertices at edge intersections in $\Gamma$ and subdividing crossing edges at these new vertices. Moreover, let $G^\top(\Gamma)$ be the planar dual of $G_\text{isc}(\Gamma)$. Clearly, $\Gamma$ has a {cell }$4$-coloring if and only if $G^\top(\Gamma)$ has a proper $4$-vertex-coloring. Since $G$ is bridgeless, so is $G_\text{isc}(\Gamma)$, and therefore $G^\top(\Gamma)$ is loopless. The $4$-Color-Theorem now implies that $\chi(G^\top(\Gamma)) \le 4$, which proves the claim.
\end{proof}

In Section~\ref{sec:2colors}, we start our investigation of {cell }colorings of graphs by characterising the drawings whose {cells} can be properly colored using only two colors. Recall that a graph is called \emph{Eulerian} if all its vertices have even degree. 

\begin{restatable}{proposition}{restateeulerian}\label{eulerian}
A drawing $\Gamma$ of a graph $G$ has a {cell }$2$-coloring if and only if $G$ is Eulerian.
\end{restatable}

For every {cell }coloring of a drawing of a non-Eulerian graph, at least $3$ colors are required.

%However, we show that this is the worst case: every graph without degree one vertices (in particular, any bridgeless graph) has numerous drawings that are $3$-colorable.
However, except for a trivial case, we show that 3 colors are the worst case: Every graph without degree 1 vertices (in particular, any bridgeless graph) has numerous drawings that are $3$-colorable. Note that, if a graph has a vertex of degree 1, then in any drawing the cell incident to that vertex touches itself. Thus, there exists no cell $k$-coloring of any drawing of that graph and any $k\in\NN$.

\begin{restatable}{proposition}{restateleaf}\label{leaf}
A graph $G$ has a drawing with a cell $3$-coloring if and only if it has no vertex of degree 1.\label{prop:leaf}
\end{restatable}

Propositions~\ref{eulerian}, and~\ref{leaf} are proved in Section~\ref{sec:2colors}.

Propositions~\ref{obvious},~\ref{eulerian}, and~\ref{leaf} motivate the problem of understanding the structure of graphs {\emph{all}} whose drawings are $3$-colorable. This leads to the following notion: If $G$ is an abstract graph, we say that $G$ is \emph{{universally cell} $3$-colorable} if every drawing of $G$ in the plane admits a {cell }$3$-coloring. 

In Section~\ref{sec:3colors} we derive several sufficient conditions for a graph to be {universally cell} $3$-colorable, and thereby draw a link between {universally cell} $3$-colorable graphs and so-called \emph{$3$-flowable} graphs, which are intensively studied in the theory of \emph{nowhere-zero flows} on graphs. For $k \in \mathbb{N}$, a \emph{nowhere-zero $k$-flow} on a graph $G$ consists of a pair $(D,f)$, where $D = (V(D), A(D))$ is an orientation of the edges of $G$, and where $f\colon A(D) \rightarrow \mathbb{Z}_k \setminus \{0\}$ is a group-valued flow on the digraph $D$, {that is}, a weighting of the arcs with non-zero group elements from $\mathbb{Z}_k$ satisfying Kirchhoff's law of flow conservation:

\begin{align*}
\forall v \in V(D)\colon \sum_{e=(w,v) \in A(D)}{f(e)}=\sum_{e=(v,w) \in A(D)}{f(e)}
\end{align*}

If a graph $G$ admits a nowhere-zero $k$-flow, we also say that $G$ is \emph{$k$-flowable}. The interest in nowhere-zero-flows stems from the following intimate connection to colorings of planar graphs. For a comprehensive introduction to the topic of nowhere-zero flows, we refer to the textbook~\cite{flows} by Zhang. Particularly relevant to the topics addressed here are the sections `Face Colorings' and `Nowhere-Zero $3$-Flows'.

\begin{theorem}[Folklore, see also \cite{flows}, Theorem 1.4.5]\label{thm:flowcoloringduality}
{If} $G$ {is} a planar graph and $\Gamma$ {is} a crossing-free embedding of $G$ in the plane, then for any $k \in \mathbb{N}$, $G$ admits a nowhere-zero $k$-flow if and only if $\Gamma$ has a cell $k$-coloring.
\end{theorem}

To provide some intuition for the correspondence between a cell - coloring and a
 nowhere\-/zero flow, let us mention that at least the `if' direction of the above equivalence can rather easily be observed as follows: Suppose that a crossing-free planar drawing $\Gamma$ of $G$ equipped with a cell $k$-coloring is given. We may identify the elements of the color set used for this coloring with the elements of the cyclic group $\mathbb{Z}_k$. Next, let $D$ be an (arbitrarily chosen) orientation of $G$, and define a $\mathbb{Z}_k$-valued flow $f$ on $D$ as follows:
 For every arc $(u,v) \in A(D)$, consider the two cells~$f_1$ and $f_2$ of the drawing $\Gamma$ which are neighboring the image of the oriented arc $(u,v)$ in $\Gamma$ to its `right' and `left', respectively (compare the example in Figure~\ref{fig:dualflow}), and define the flow value of $f$ on $(u,v)$ as $f(u,v):=c(f_1)-c(f_2)$, where $c(f_i) \in \mathbb{Z}_k$ for $i=1,2$ denote the colors assigned to the {cell}s. Since $f_1$ and $f_2$ are neighboring cells in $\Gamma$, the condition on a proper cell-coloring yields that with this definition of $f$, we have $f(e) \neq 0$ for all arcs $e \in E(D)$. In addition, it is easy to see from the definition of $f$ that the sum of out- and in-flow around any vertex in $D$ equals $0$, hence, $(D,f)$ indeed forms a nowhere zero $k$-flow. 

\begin{figure}[ht]
\centering  
\includegraphics[scale=0.6]{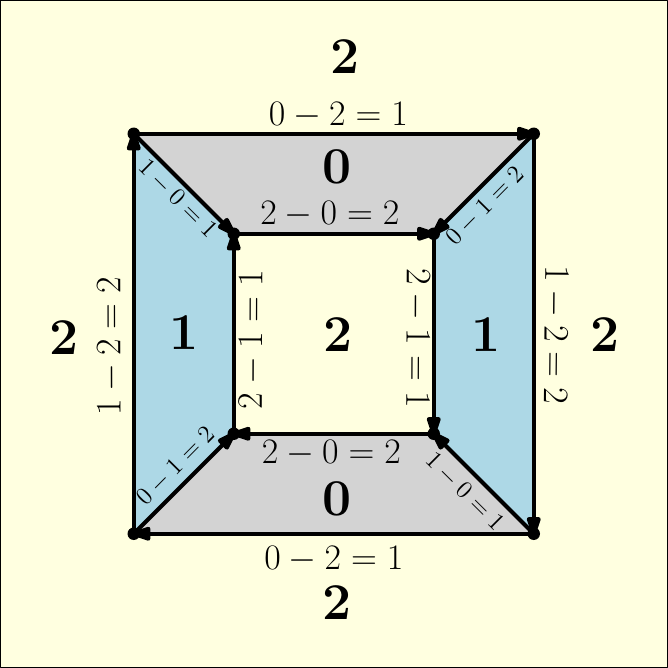}
\caption{A nowhere zero-$3$-flow of the cube graph obtained from a cell $3$-coloring of a planar drawing.}\label{fig:dualflow}
\end{figure}

Similar to the situation for {cell }colorings, only bridgeless graphs can have nowhere-zero flows, as the flow value of a bridging edge must be $0$. Conversely, a {well-known} result by Seymour~\cite{6flows} states that every bridgeless graph admits a nowhere-zero $6$-flow.
The following result relates the existence of nowhere-zero $3$-flows in graphs with the existence of {cell }$3$-colorings for all their drawings.

\begin{restatable}{theorem}{restatethreeflowsuff}\label{thm:3flowsuff}
	{Every graph admitting a nowhere-zero $3$-flow is universally cell $3$-colorable.}
\end{restatable}

{From Theorem~\ref{thm:flowcoloringduality} and Theorem~\ref{thm:3flowsuff} we can see that a planar graph is universally cell $3$-colorable if and only if it admits a nowhere zero $3$-flow. This has the following interesting consequence regarding the computational complexity of recognising universally cell $3$-colorable graphs.

\begin{corollary}
Deciding whether a given planar graph is universally cell $3$-colorable is \NP -complete.\label{planarhard}
\end{corollary}

\begin{proof}
Testing whether a planar graph is {universally cell} $3$-colorable by Theorem~\ref{K33free} is equivalent to testing whether a given planar graph is $3$-flowable. This problem is clearly contained in the class \NP\ (we can verify a modulo-$3$-orientation in polynomial time). By Theorem~\ref{thm:flowcoloringduality} and planar duality, we can furthermore reduce the problem of deciding whether a given planar graph is properly $3$-vertex-colorable to this problem. Since this problem is \NP -complete (see \cite{garey}), we deduce the claim.
\end{proof}

}

Based on Gr\"{o}tzsch's theorem, we obtain another interesting positive result.

\begin{restatable}{corollary}{restatefourconn}\label{4conn}
Every $4$-edge-connected graph is {universally cell} $3$-colorable.
\end{restatable}

The proofs of Theorem~\ref{thm:3flowsuff} and {Corollary}~\ref{4conn} are presented in Section~\ref{sec:3colors}.
Looking at Theorem~\ref{thm:3flowsuff}, it is natural to ask whether there are {universally cell} $3$-colorable graphs that are not $3$-flowable. In Section~\ref{sec:infinite}, we answer this question in the positive by providing an infinite family of graphs with this property. For every $n \in \mathbb{N}$, {$K_{3,n-3}^+$} denotes the {$n$-vertex-}graph obtained from the complete bipartite graph {$K_{3,n-3}$} by joining two vertices in the partite set of size $3$ by an edge.

\begin{restatable}{theorem}{restateKthreens}\label{K3ns}
For every $n \ge 7$, the graph {$K_{3,n-3}^+$} is {universally cell} $3$-colorable but does not admit a nowhere-zero $3$-flow.
\end{restatable}

However, examples of graphs as given by Theorem~\ref{K3ns} are rarely spread. In fact, Sudakov~\cite{sudakov3flows} proved that random graphs expected to have minimum degree at least $2$ are expected to have a nowhere-zero $3$-flow, thereby establishing in a strong sense that almost all graphs admit a nowhere-zero $3$-flow. In Section~\ref{sec:mincex} we prove the following two results, which show that an equivalence between {universal cell} $3$-colorability and the existence of nowhere-zero $3$-flows {(as we have it for planar graphs by Theorems~\ref{thm:flowcoloringduality} and~\ref{thm:3flowsuff})} holds at least for sparse graph classes beyond planar graphs. The proofs rely on Propositon~\ref{subcontclosed}, which shows that {universal cell} $3$-colorability is hereditary with respect to subcontractions{.} (For a definition, see the paragraph at the end of this section.)\oldV{.}

\begin{restatable}{theorem}{restatecubic} \label{scubic}
A graph with maximum degree at most $3$ is {universally cell} $3$-colorable if and only if it is $3$-flowable.
\end{restatable}

\begin{restatable}{theorem}{restatekthreethreefree}\label{K33free}
A $K_{3,3}$-minor-free graph is {universally cell} $3$-colorable if and only if it is $3$\=/flowable.
\end{restatable}

We have not been able to find graphs which are {universally cell} $3$-colorable but not $3$-flowable except for graphs arising by simple operations from the examples given by Theorem~\ref{K3ns}. To be more precise, we believe that excluding the graphs $K_{3,n-3}^+$ for $n \ge 7$ as subcontractions could already be sufficient to yield an equivalence between {universal cell} $3$-colorability and $3$-flowability.

\begin{conjecture}\label{main}
If $G$ is a {universally cell} $3$-colorable graph which does not have a subcontraction isomorphic to $K_{3,n-3}^+$ for some {$n \ge 7$}, then $G$ is $3$-flowable.
\end{conjecture}

Interestingly, a positive answer to Conjecture~\ref{main} would also imply a positive answer to the following long-standing \bothV{Conjecture}{conjecture} by Tutte. 

\begin{conjecture}[Tutte's $3$-Flow-Conjecture, Conjecture~1.1.8 in \cite{flows}]\label{tutte}
Every $4$-edge-connected graph admits a nowhere-zero $3$-flow.
\end{conjecture}

\bothV{
To see this, suppose Conjecture~\ref{main} holds true, and let $G$ be a given $4$-edge-connected graph. By Theorem~\ref{4conn}, $G$ is facially $3$-colorable. Since $G$ is $4$-edge-connected, so is each of its subcontractions. Since each $K_{3,n}^+$ has a vertex of degree $3$, this means that $G$ has no subcontraction isomorphic to a $K_{3,n}^+$ with $n \ge 4$. Hence, Conjecture~\ref{main} yields that $G$ is $3$-flowable, as claimed in Tutte's conjecture.}
{\begin{claim}
		If Conjecture~\ref{main} holds true, then also Conjecture~\ref{tutte} holds true.
	\end{claim}
	
	\begin{proof}
Suppose Conjecture~\ref{main} holds true, and let $G$ be any given $4$-edge-connected graph. By Corollary~\ref{4conn}, $G$ is universally cell $3$-colorable. Since $G$ is $4$-edge-connected, so is each of its subcontractions. Since each $K_{3,n-3}^+, n \ge 7$ has a vertex of degree $3$, this means that $G$ has no subcontraction isomorphic to a $K_{3,n-3}^+$ with $n \ge 7$. Hence, Conjecture~\ref{main} yields that $G$ is $3$-flowable, as claimed in Tutte's conjecture.
\end{proof}
}

In Section~\ref{sec:mincex} we obtain several properties {that must be fulfilled by} smallest counterexamples to Conjecture~\ref{main}\oldV{ must have (Theorem~\ref{thm:mincex})}, and thereby limit the class of graphs for which the \bothV{C}{c}onjecture has to be checked{ (Theorem~\ref{thm:mincex})}.

We conclude with an open question concerning computational complexity in Section~\ref{sec:schluss}. 
\subsubsection*{Related work}
%This paragraph is devoted to a short overview of related results on colorings of drawings of graphs that we could find in the literature. 

The only previous work on the notion of coloring discussed in this paper (namely that colors are assigned to the cells of a drawing) we could find in the literature is solely about the planar case of this problem, that is, the problem to determine the chromatic number of the plane dual graph of a given crossing-free plane drawing of a graph. The interpretation of the 4-Color-Theorem as a result of coloring regions of the plane, albeit the original one, was not the one that caught on as much as the equivalent one of coloring vertices of a graph.
However, because any coloring result on planar graphs can be dualized as mentioned above, any such result is somewhat related to what we do, but nothing really stands out. For the reader interested in this field of research, we confer to the somewhat recent survey of Borodin \cite{placol}.

Another concept that was developed to better understand the 4-color-problem is the one of \emph{Nowhere-zero $k$-flows}, which are also used in our paper, even if mostly in the case $k=3$. This area of research is also still quite active, see Zhang \cite{flows} for a book that covers the essentials. A recent paper, which is particularly close since it also deals with Eulerian graphs, is the paper of  M\'a\v cajov\'a and \v Skoviera \cite{eulerianflows}. For the state of the art concerning Tutte's 3-Flow Conjecture, we refer to Lov\'asz, Thomassen, Wu and Zhang \cite{6connectivity}.
%In that sense, all results known about bounds on the chromatic number of subclasses of the planar graphs, via duality, carry over to a dual statement in the setting of cell-coloring planar drawings. Two examples: Every bipartite planar graph is $2$-colorable, hence, by duality, every crossing-free drawing of an Eulerian graph is cell-$2$-colorable. Similarly, every simple planar cubic graph different from $K_4$, by Brook's theorem, is $3$-colorable, and via duality, this means that every plane triangulation different from $K_4$ is cell $3$-colorable. As it would be impossible to make a complete list, we leave it at these two examples in this overview, stating clearly that other results of a similar type could be derived, simply by dualizing results known about the chromatic number of planar graphs.  

Other previous work on coloring drawings of possibly non-planar graphs we could find in the literature mostly deals with a different notion, in which colors are assigned to the \emph{edges} of a graph, and the condition is that no two edges which intersect in the drawing may be assigned the same color. This notion of coloring seems to have been studied first by Sinden\noldV{ in}~\cite{sinden}, who studied such colorings of rectilinear drawings of the complete bipartite graphs, motivated by a vehicle scheduling problem. The same notion of coloring was studied later on in a series of papers concerning colorings of rectilinear drawings of the complete graphs, compare~\cite{aich,biniaz,bose,dujmovic}. The fewest number of colors required for such a crossing-free edge-coloring, minimized over all possible (rectilinear) drawings of a graph, is also known as the \emph{geometric thickness} of the graph.

\subsubsection*{Notation and important definitions}

The graphs considered in this paper are finite multi\oldV{-}graphs. For a graph without loops and parallel edges, we use the term \emph{simple} graph. {An edge whose deletion increases the number of connected components is called a \emph{bridge}. Consequently, a graph that does not contain a bridge is called \emph{bridgeless}.} We write $e=uv$ for an edge $e$ in an undirected graph to indicate that $u$ and~$v$ are the endpoints of $e$, and $a=(u,v)$ for an arc in a directed graph to indicate that $a$ has tail~$u$ and head $v$. Given an undirected graph $G$, we denote by $V(G)$ its vertex\bothV{-}{ set} and by~$E(G)$ its (multi\=/)edge\bothV{-}{ }set. Similarly, for an orientation $D$ of $G$ we denote by $V(D)=V(G)$ its vertex\bothV{-}{ }set and by $A(D)$ its (multi\=/)arc\bothV{-}{ }set\bothV{, and w}{, which contains exactly one arc (directed in one of the two possible directions), for each edge of $G$. W}e say that $G$ is the \emph{underlying graph} of $D$. For a graph $G$ and a vertex $v \in V(G)$, we denote by $d_G(v)$ its degree in $G$, which is the number of incident edges of $v$, where loops are counted with multiplicity $2$. Similarly, if $D$ is a digraph and~\mbox{$v \in V(D)$} then by $d_D^+(v)$ ($d_D^-(v)$) we denote the out-degree (in-degree) of $v$ in $D$, \bothV{i.e.}{that is}, the number of out-arcs (in-arcs) incident with $v$.\nbothV{ We denote the}{ The} subgraph of $G$ induced by a vertex set $U\subseteq V(G)$ {is denoted }by $G[U]=(U,\{uv\in E(G)\mid u,v\in U\})$.

Let $G$ and $G'$ be graphs. 
We say that $G'$ is a \emph{minor} of $G$, if $G'$ is isomorphic to a graph obtained from $G$ via a sequence of finitely many edge contractions, edge deletions and vertex deletions. Given a vertex set $X \subseteq V(G)$, we denote by $G/X$ the graph obtained from $G$ by adding a single vertex $v_X$, adding an edge between $u$ and $v_X$ to $E(G/X)$ for every edge $uv \in E(G)$ with $u \notin X, v \in X$, then deleting all vertices in $X$. We say that the multi-graph $G/X$ is obtained from $G$ by \emph{identifying} $X$ into $v_X$. We say that $G'$ is a \emph{subcontraction} of $G$ if there is a sequence $G=G_0, G_1, \ldots, G_\ell\cong G'$ of graphs {(the symbol ``$\cong$'' meaning ``isomorphic to'')} such that for every $i \in \{1,\ldots,\ell\}$ there is $X_i \subseteq V(G_{i-1})$ such that $G_i=G_{i-1}/X_i$. Note that every subcontraction of $G'$ is a subcontraction of $G$ as well.

By a \emph{drawing} $\Gamma$ of a graph $G$ we mean an immersion of $G$ in the plane such that vertices are mapped to distinct points and edges are represented by continuous curves connecting the images of their endpoints (closed curves in the case of loops), but which do not contain any other vertices. Edges may self-intersect and pairs of edges may intersect, but there are only finitely many points of intersection. If required, for $e \in E(G)$ we will use the notation $\gamma(e)$ to indicate the set of points on the curve representing $e$, and $\gamma^\circ(e)$ for the set of interior points of $\gamma(e)$ (distinct from the images of the endpoints of $e$). {We also call $G$ the \emph{underlying graph} of the drawing $\Gamma$.} A more restricted class of drawings are the \emph{good drawings}. They are defined as the drawings satisfying the following additional properties:
\begin{itemize}
\item \bothV{no point is contained}{no point in the plane is contained} in the interiors of more than two edges,
\item every non-loop edge is free of self-intersections and loops do not self-intersect in their interior,
\item every two adjacent edges intersect only in their common endpoints,
\item non-adjacent edges intersect in at most one common point, which is a proper crossing, and
\item loops do not intersect other edges at all.
\end{itemize}
Good drawings of simple graphs have been studied extensively in the literature, mainly because the crossing number of a graph is attained by good drawings. We will go a little further into this in Section \ref{sec:3colors}\bothV{,}{;} see also \cite{crossingnumbers} for a survey on this topic. \oldV{In case $\Gamma$ is a drawing of $G$, we also say that $G$ is the \emph{underlying graph} of $\Gamma$.}

By $\mathcal{F}(\Gamma)$ we denote the set of the \emph{\bothV{faces}{cells}} of $\Gamma$, \bothV{i.e.}{that is}, the connected components of $\mathbb{R}^2-\Gamma$. Note that there is always a unique unbounded \bothV{face}{cell} \bothV{which surrounds}{surrounding} the whole drawing, which we refer to as the \emph{outer \bothV{face}{cell}}. Placing an additional vertex at every inner intersection of at least two edges in a drawing $\Gamma$ and making
\bothV{two such vertices adjacent}{these vertices adjacent to all (new or original) vertices} whenever they appear consecutively on the same edge of the drawing, we obtain an embedded planar graph $G_\text{isc}(\Gamma)$, which we refer to as the \emph{planarization} of $\Gamma$. See \bothV{Figure}{Fig.}~\ref{fig:wildm} for an example. We refer to the vertices of $G_\text{isc}(\Gamma)$ situated at intersections of edges in $\Gamma$ as \emph{intersection vertices}. The vertices which already appear in $\Gamma$ are called \emph{normal vertices}. Note that $G_\text{isc}(\Gamma)$ captures the most relevant combinatorial properties of the drawing $\Gamma$. In particular, the cell decompositions of $\mathbb{R}^2-\Gamma$ and $\mathbb{R}^2-G_\text{isc}(\Gamma)$ are isomorphic. By $G^\top(\Gamma)$ we will denote the planar dual graph of~$G_\text{isc}(\Gamma)$, \bothV{i.e.}{that is}, whose vertices correspond to the faces of $G_\text{isc}(\Gamma)$ (or $\Gamma$), two of which have a connecting edge for every shared boundary-edge in $G_\text{isc}(\Gamma)$. The dual graph $G^\top(\Gamma)$ has a natural planar embedding in which the vertices are placed inside the corresponding \bothV{faces}{cells} of $\Gamma$ and every dual edge crosses its corresponding boundary edge, see \bothV{Figure}{Fig.}~\ref{fig:wildr} for an example.

\begin{figure}[ht]
	\centering
	\begin{subfigure}{0.34\textwidth}
		\centering
		\includegraphics[scale=0.7,page=1]{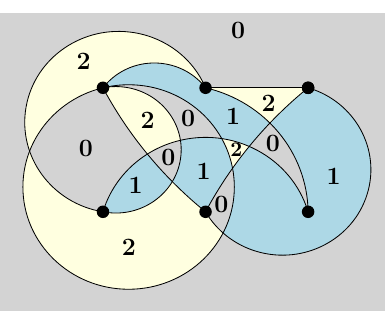}
		\caption{Drawing $\Gamma$ of $G$ colored by $c$}
		\label{fig:wildl} 
	\end{subfigure}
	\begin{subfigure}{0.31\textwidth}
		\centering
		\includegraphics[scale=0.7,page=2]{wild_drawing.pdf}
		\caption{The planarization $G_\text{isc}(\Gamma)$}
		\label{fig:wildm} 
	\end{subfigure}
	\begin{subfigure}{0.3\textwidth}
		\centering
		\includegraphics[scale=0.7,page=3]{wild_drawing.pdf}
		\caption{$G^\top(\Gamma)$, dual to $G_\text{isc}(\Gamma)$}
		\label{fig:wildr} 
	\end{subfigure}
	\caption{\bothV{Face-}{Cell }3-coloring $c$ of a drawing of a multigraph $G$, associated graphs $G_\text{isc}$ and $G^\top$}
	\label{fig:wild}
\end{figure}
 
\section{Existence of 2- and 3-Colorable Drawings}
\label{sec:2colors}

In this section we derive circumstances under which drawings of graphs are \bothV{facially}{cell }$2$- or $3$-colorable, leading to the proofs of Propositions~\ref{eulerian} and~\ref{leaf}.

\restateeulerian*

\begin{proof}
Let $\Gamma$ be a drawing of the graph $G$. Suppose for the first direction that $\Gamma$ has a \bothV{face-}{cell }$2$-coloring \mbox{$c\colon\mathcal{F}(\Gamma) \rightarrow \{0,1\}$}. Then for any vertex $v$ in the drawing, the incident \bothV{faces}{cells} in cyclical order around the vertex have to alternate between the colors $0$ and $1$. Thus, the degree of $v$ must be even. Hence, $G$ is indeed Eulerian.

Suppose vice versa that $G$ has only vertices of even degree, and without loss of generality assume that $G_\text{isc}(\Gamma)$ is connected. Note that every intersection vertex in $G_{\text{isc}}(\Gamma)$ has even degree because every edge of the drawing passing through the corresponding intersection point contributes two edge-segments incident with the intersection-vertex. Hence, $G_\text{isc}(\Gamma)$ is Eulerian as well. As is well-known~\cite{welsh}, this means that the planar dual graph $G^\top(\Gamma)$ is bipartite. The $2$-coloring of $G^\top(\Gamma)$ now yields a proper \bothV{face-}{cell }$2$-coloring of $G_\text{isc}(\Gamma)$ respectively {a proper cell $2$-coloring of} $\Gamma$.
\end{proof}

Using {the fact} that drawings of Eulerian graphs are \bothV{face-}{cell }$2$-colorable, we can prove that every bridgeless graph admits a \bothV{face-}{cell }$3$-colorable drawing. {An illustration of the proof is provided in Fig.~\ref{outerdraw}.}

\begin{figure}[ht]
\centering  
\begin{subfigure}[t]{.45\textwidth}
\begin{flushright}
\includegraphics[width=57mm,page=1]{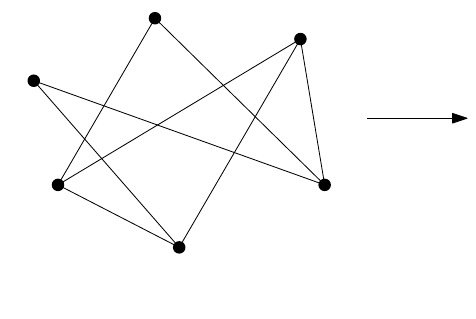}
\caption{Outer Drawing $\Gamma$ of graph $G$}
\label{fig:3col1}
\end{flushright}
\end{subfigure}
\begin{subfigure}[t]{.45\textwidth}
\includegraphics[width=57mm, page=2]{3colorabledrawings}
\caption{Drawing $\Gamma^+$ of the Eulerian graph $G^+$}
\label{fig:3col2}
\vspace*{1em}
\end{subfigure}
\begin{subfigure}[b]{.45\textwidth}
\begin{flushright}
\includegraphics[width=57mm, page=4]{3colorabledrawings}
\caption{Corresponding \bothV{face-}{cell }$3$-coloring of $\Gamma$}
\label{fig:3col4}
\end{flushright}
\end{subfigure}
\begin{subfigure}[b]{.45\textwidth}
\includegraphics[width=57mm, page=3]{3colorabledrawings}
\caption{\bothV{Face-}{Cell }$2$-coloring of $\Gamma^+$\hspace{30pt}$\,$}
\label{fig:3col3}
\end{subfigure}
\caption{Illustration of the argument in the proof of Proposition~\ref{outerface}, to be read \subref{fig:3col1}, \subref{fig:3col2}, \subref{fig:3col3}, then \subref{fig:3col4}.}\label{outerdraw}
\end{figure}

\begin{restatable}{proposition}{restateouterface}\label{outerface}
Every drawing $\Gamma$ of a bridgeless graph $G$ in which all vertices lie on the outer \bothV{face}{cell} has a \bothV{face-}{cell }$3$-coloring.
\end{restatable}

\begin{proof}
Let $G$ be a bridgeless graph and let $\Gamma$ be a drawing of $G$ in which every vertex is incident to the outer \bothV{face}{cell}. Let $O \subseteq V(G)$ be the set of vertices of odd degree in $G$. Let $G^+$ be the graph obtained from $G$ by adding a new vertex $v_O$ and all edges from $v_O$ to $v$ for $v \in O$. This way, we achieve that $G^+$ is an Eulerian graph (note that, by the handshake\bothV{-}{ }lemma, $d(v_O)=|O|$ must be even). Furthermore, since $\Gamma$ is a drawing in which all vertices in $O$ are incident to the outer \bothV{face}{cell}, we can obtain a drawing $\Gamma^+$ of $G^+$ by placing $v_O$ within the outer \bothV{face}{cell} of the drawing~$\Gamma$ and connecting it to the vertices in $O$ on the outer \bothV{face}{cell} without introducing any new edge-intersections (see \bothV{Figure}{Fig.}~\ref{outerdraw} for an illustration). By Proposition~\ref{eulerian}, $\Gamma^+$ now admits a \bothV{face-}{cell }$2$-coloring $c\colon\mathcal{F}(\Gamma^+)\rightarrow \{0,1\}$. Note that every interior \bothV{face}{cell} of the drawing $\Gamma$ is also an interior \bothV{face}{cell} of the drawing $\Gamma^+$, and that two interior \bothV{faces}{cells} of $\Gamma$ are adjacent if and only if they are in $\Gamma^+$. We now define a $3$-coloring of the \bothV{faces}{cells} of $\Gamma$ by assigning color $c(f) \in \{0,1\}$ to each interior \bothV{face}{cell} $f$ of~$\Gamma$ and by coloring the outer \bothV{face}{cell} of $\Gamma$ with color $2$. By definition of $c$, any two adjacent interior \bothV{faces}{cells} have a different color and the outer \bothV{face}{cell} has a color different from any other \bothV{face}{cell}. Since $G$ is bridgeless, the outer \bothV{face}{cell} does not touch itself. Thus, in total, $\Gamma$ is \bothV{face-}{cell }$3$-colorable.
\end{proof}

While we mostly deal with bridgeless graphs in this paper, the \emph{existence} of a \bothV{face-}{cell }3-coloring of any drawing of a graph does not imply bridgelessness. In \bothV{Figure}{Fig.}~\ref{fig:noleafconv}, you can see a drawing of a graph with a bridge, which has a \bothV{face-}{cell }3-coloring.

\begin{figure}[ht]
\centering  
%\begin{subfigure}{0.4\textwidth}
%\centering  
%\includegraphics[scale=0.6,page=1]{noleafconvex.pdf}
%\end{subfigure}
%\begin{subfigure}{0.4\textwidth}
%\centering  
%\includegraphics[scale=0.6,page=2]{noleafconvex.pdf}
%\end{subfigure}
\includegraphics[scale=0.8]{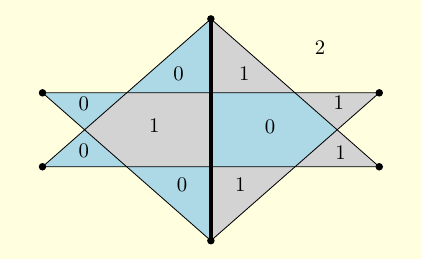}
\caption{A drawing of a graph $G$ with a bridge with a \bothV{face-}{cell }3-coloring: the bridge of $G$ is fat.}\label{fig:noleafconv}
\end{figure}

\restateleaf*

\begin{proof}
{An illustration of the key arguments of this proof is provided in Fig.~\ref{fig:prop3fig}.}
{
\begin{figure}[ht]
\centering  
\includegraphics[scale=.9]{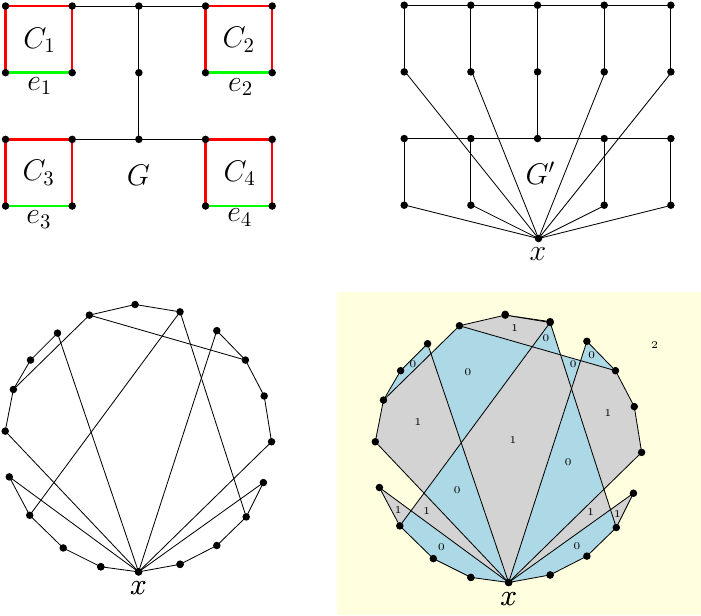}
\caption{Illustration of the proof of the Proposition~\ref{leaf}. Top left: A graph $G$ with no vertices of degree~$1$. A maximal collection of vertex-disjoint cycles is marked fat and red, and from each cycle an edge is selected and highlighted green. Top right: The bridgeless graph $G'$ obtained from $G$, containing the new vertex $x$. Bottom left: A drawing of $G'$ in which all vertices lie on the outer cell. Bottom right: A drawing of $G$ obtained from this drawing of $G'$, equipped with a cell $3$-coloring.}\label{fig:prop3fig}
\end{figure}}

It is not hard to see that, if a graph has a vertex of degree 1, the \bothV{face}{cell} incident to that vertex touches itself at its edge. We are therefore left to show that graphs $G$ with no \bothV{leafs}{leaves}, that is, without vertices of degree 1, have a drawing which is \bothV{face-}{cell }$3$-colorable. 

{Let $G$ be such a graph.}
We now apply the following procedure to generate a drawing $\Gamma$ of $G$ that introduces some intersections making the resulting planar graph $G_\text{isc}(\Gamma)$ bridgeless: starting with $G$, as long as the remaining graph has a cycle, remove all vertices of that cycle and set it aside. What we end up with is a family $C_1,...,C_k$ of cycles and a remainder of vertices, which induces a forest $F$. Construct the graph $G'$ from $G$ in the following way: For any $i\in\{1,...,k\}$ choose an edge $e_i\in C_i$, then add a vertex $x$, which is adjacent to the endpoints of all edges $e_i$, then delete these edges. This gives rise to a graph $G'$ with the property that every drawing of it is a drawing of $G$ in which the edges $e_i$ all intersect at least once, namely at the location of $x$. To prove that some of these drawings admit a \bothV{face-}{cell }3-coloring, by Proposition \ref{outerface} it suffices to prove that $G'$ is bridgeless. Quite obviously, its subgraph $G'[\{x\}\cup\bigcup_{i=1}^k V(C_k)]$ is 2-edge-connected, so it suffices to prove that $G'-e$ is connected for every edge $e \in E(G')$ at least one of whose endpoints lies in $V(F)$. We will do this by showing that both endpoints of such an edge $e$ can either reach each other or reach a vertex in $\{x\}\cup\bigcup_{i=1}^k V(C_k)$. Let $e=uv$ be such an edge and consider the endpoint $v$ of $e$. If $v \in \{x\}\cup\bigcup_{i=1}^k V(C_k)$, our claim holds trivially{ for $v$}. Otherwise, we have $v \in V(F)$. Consider a longest path $P=v,v_1,\ldots,v_\ell$ in $F-e$ starting at $v$ {(if no path of positive length in $F-e$ starting at $v$ exists, then $P$ is defined to consist only of the single vertex~$v$)}. If $v_\ell=u$, then we have found a path in $G'-e$ connecting $u$ and $v$. Otherwise, since~$v_\ell$ has degree at least two in $G$, and, thus, in $G'$ as well, it must be incident to an edge $f$ in $G'$ distinct from the last edge of $P$. Let $v'$ be the other end of $f$. Since $F$ contains no cycles, we must have $v' \notin \{v,v_1,\ldots,v_\ell\}$, and since $P$ is longest, we must have $v' \notin V(F)$. Hence, $P+f$ forms a path connecting $v$ to $\{x\}\cup\bigcup_{i=1}^k V(C_k)$ in $G'-e$, which again yields the desired claim. This shows that, indeed, $G'$ is $2$-edge-connected and concludes the proof.
\end{proof}

\section{Sufficient Conditions for \bothV{Facial}{Universal Cell} 3-Colorability}\label{sec:3colors}

In this section we show conditions \bothV{on}{for} a graph $G$ that guarantee that every drawing of $G$ is \bothV{face-}{cell }$3$-colorable, leading to the proofs of Theorem~\ref{thm:3flowsuff} and \bothV{Theorem}{Corollary}~\ref{4conn}.
The following well-known equivalence between nowhere-zero $3$-flows and special orientations of graphs will be used frequently in our study of the relationships between nowhere-zero $3$-flows and \bothV{facially}{universally cell} $3$-colorable graphs. 

\begin{definition}
Let $G$ be a graph. An orientation $D$ of $G$ is called \emph{modulo-$3$-orientation}, if for any vertex $v \in V(D)$, the \emph{excess} $\text{exc}_D(v)\coloneqq d_D^+(v)-d_D^-(v)$ at $v$ is divisible by $3$.
\end{definition}

\begin{lemma}[Folklore, see also \cite{flows}, Lemma 4.1.2]\label{mod3or}
A graph $G$ is $3$-flowable if and only if it admits a modulo-$3$-orientation.
\end{lemma}

We are now ready to prove Theorem~\ref{thm:3flowsuff}. The rough idea is illustrated by an example in \bothV{Figure}{Fig.}~\ref{fig:K33}.

\restatethreeflowsuff*

\begin{figure}[ht]
\centering
\begin{subfigure}{0.33\textwidth}
\centering
\includegraphics[scale=0.75]{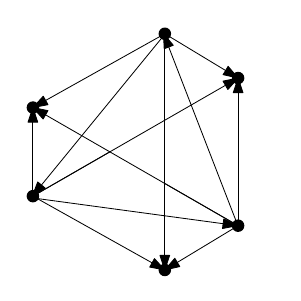}
\caption{Mod.-$3$-orientation of $K_{3,3}$}
\label{fig:K33l} 
\end{subfigure}
\begin{subfigure}{0.36\textwidth}
\centering
\includegraphics[scale=0.75]{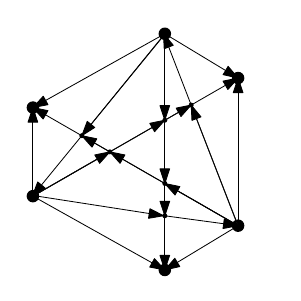}
\caption{Mod.-$3$-orientation of $G_\text{isc}(\Gamma)$}
\label{fig:K33m} 
\end{subfigure}
\begin{subfigure}{0.29\textwidth}
\centering
\includegraphics[scale=0.75]{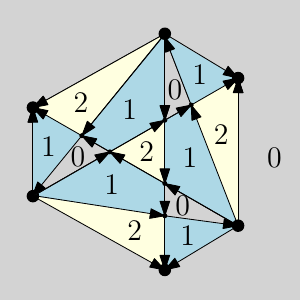}
\caption{Induced \bothV{face-}{cell }3-coloring $c$}
\label{fig:K33r} 
\end{subfigure}
\caption{\bothV{Face-}{Cell }3-coloring $c$ of a drawing $\Gamma$ of $K_{3,3}$ given by Theorem~\ref{thm:3flowsuff}. {The coloring $c$ of the drawing can be obtained from the modulo $3$-orientation with the following process (we refer to{ Zhang}~\cite{flows}, Chapter ``flows and face colorings'' for a proof of the correctness of this process): We use the color-set $\mathbb{Z}_3\simeq \{0,1,2\}$. We start by picking a face (say, the outer face) and assign color $0$ to it. Now, until the coloring is complete, we pick a face $f_1$ in the drawing that has not yet been colored but shares an edge~$e$ with an already colored face $f_2$. We set $c(f_1):=c(f_2)+1$ (summation modulo $3$) if $f_1$ is to the right of the directed edge $e$ in the modulo $3$-orientation of $G_\text{isc}(\Gamma)$, and $c(f_1):=c(f_2)-1$ otherwise. The properties of the modulo $3$-orientation can be used to show that this indeed always yields a proper cell $3$-coloring.}}
\label{fig:K33}
\end{figure} 
\begin{proof}
Let $G$ be a $3$-flowable graph. By Lemma~\ref{mod3or}, there is a modulo-$3$-orientation $D$ of $G$. Let~$\Gamma$ be a given drawing of $G$. In order to show that $\Gamma$ is $3$-colorable, it suffices to show that the embedded planar graph $G_\text{isc}(\Gamma)$ has a \bothV{face-}{cell }$3$-coloring. By Theorem~\ref{thm:flowcoloringduality} and Lemma~\ref{mod3or}, we can do this by constructing a modulo-$3$-orientation $D_\text{isc}(\Gamma)$ of $G_\text{isc}(\Gamma)$. For every oriented edge $(u,v)$ in the orientation $D$ of $G$, consider the corresponding curve in the drawing $\Gamma$. In $G_\text{isc}(\Gamma)$, this curve corresponds to a trail $u=x_0,x_1,\ldots,x_\ell=v$ where every interior vertex $x_i$ is an intersection vertex. We now define the orientation of the edges $x_{i-1}x_i, i=1,\ldots,\ell$ in $D_\text{isc}(\Gamma)$ by orienting this trail from $u$ towards $v$, \bothV{i.e.}{that is}, we have $(x_{i-1},x_i) \in A(D_\text{isc}(\Gamma))$ for all $i$. With this, we uniquely assign orientations to all the edges of $G_\text{isc}(\Gamma)$. It remains to be seen why $D_\text{isc}(\Gamma)$ is a modulo-$3$-orientation of $G_\text{isc}(\Gamma)$. For this, let $x \in V(D_\text{isc}(\Gamma))$ be arbitrary. If $x$ is a normal vertex, then clearly $\text{exc}_{D_\text{isc}(\Gamma)}(x)=\text{exc}_{D}(x)$, and so the excess at this vertex is divisible by $3$ as required. If $x$ is an intersection vertex, every (closed) directed trail in $D_{\text{isc}}(\Gamma)$ induced by an edge of $\Gamma$ through the intersection point at $x$ has to enter and leave the point $x$ the same number of times. Hence, we have $\text{exc}_{D_{\text{isc}}(\Gamma)}(x)=d_{D_{\text{isc}}(\Gamma)}^+(x)-d_{D_{\text{isc}}(\Gamma)}^-(x)=0$. We conclude that the excess of every vertex in the orientation is divisible by $3$, and so $D_\text{isc}(\Gamma)$ defines a modulo-$3$-orientation of $G_\text{isc}(\Gamma)$. Since $\Gamma$ was arbitrary, this concludes the proof.
\end{proof}

We have the following interesting consequence of Theorem~\ref{thm:3flowsuff}, which is useful in order to limit the complexity of the drawings which have to be checked to certify \bothV{facial}{universal cell} $3$-colorability.

\begin{corollary}\label{cor:crossingsmakeiteasier}
Let $G$ be a graph, and let $\Gamma^\ast$ be a \bothV{face-}{cell }$3$-colorable drawing of $G$. \nbothV{Let $\Gamma'$ be}{If $\Gamma'$ is} a drawing of $G_\text{isc}(\Gamma^\ast)$\bothV{ and}{, then $\Gamma'$ is cell $3$-colorable. Further,} \nbothV{let $\Gamma$ be}{if $\Gamma$ is} a drawing of $G$ such that $\mathbb{R}^2-\Gamma$ and $\mathbb{R}^2-\Gamma'$ induce the same cell-decompositions of the plane, that is, $G^\top(\Gamma)=G^\top(\Gamma')$\nbothV{. T}{, t}hen also $\Gamma$ has a \bothV{face-}{cell }$3$-coloring.
\end{corollary}

\begin{proof}
Since $\Gamma^\ast$ is \bothV{face-}{cell }$3$-colorable, so is $G_\text{isc}(\Gamma^\ast)$. Now Theorem~\ref{thm:flowcoloringduality} implies that $G_\text{isc}(\Gamma^\ast)$ is $3$-flowable. By Theorem~\ref{thm:3flowsuff} this means that $G_\text{isc}(\Gamma^\ast)$ is \bothV{facially}{universally cell} $3$-colorable. Thus $\Gamma'$ is \bothV{face-}{cell }$3$-colorable. This clearly means that also $\Gamma$ is \bothV{face-}{cell }$3$-colorable.
\end{proof}

A well-known fact in crossing number theory is that every drawing of a graph can be reduced by a set of local uncrossing-operations to a good drawing. This implies the following result. The proof is standard, but lengthy, hence we defer it to the appendix. A similar proof (which however only deals with simple graphs) can be found in{ the book of Schaefer}~\cite{crossingnumbers}, Lemma 1.3.

\begin{restatable}{proposition}{restatetopssuffice}\label{prop:topssuffice}
Let $G$ be a graph. Then $G$ is \bothV{facially}{universally cell} $3$-colorable if and only if all \emph{good} drawings of $G$ admit \bothV{face-}{cell} $3$-colorings.
\end{restatable}

We now proceed to prove \bothV{Theorem}{Corollary}~\ref{4conn}. To do so, we make use of the following result of Gr\"{o}tzsch.

\begin{theorem}[Gr\"{o}tzsch's Theorem, see \cite{grotzsch}]
Every triangle-free loopless planar graph is properly $3$-vertex-colorable.
\end{theorem}

Using duality, this translates into the following reformulation.

\begin{theorem}\label{dualgroetzsch}
Every $4$-edge-connected planar graph admits a nowhere-zero $3$-flow.
\end{theorem}

\begin{proof}
Let $G$ be a $4$-edge-connected planar graph{ and let $\Gamma$ be a planar embedding of $G$}. Note that this means that the planar dual graph of \bothV{$G$}{$\Gamma$} is simple and triangle-free. {This can be seen as follows: Suppose towards a contradiction that the dual graph of $\Gamma$ contains a cycle $C$ of length at most $3$. Then the set of edges in $G$ corresponding to the dual edges contained in $C$ would form an edge\noldV{-}cut in $G$ of size at most $3$, separating the vertices of $G$ whose dual faces lie in the interior of $C$ from those whose dual faces lie in the exterior of $C$. This would however contradict the $4$-edge-connectivity of $G$. Therefore, the dual graph of $\Gamma$ is indeed simple and triangle-free.}

Hence, Gr\"{o}tzsch's Theorem implies that the dual graph of \bothV{$G$}{$\Gamma$} admits a proper $3$-vertex-coloring, and therefore the \bothV{faces}{cells} of \bothV{$G$}{$\Gamma$} can be properly colored with \bothV{$3$}{three} colors. The claim now follows from Theorem~\ref{thm:flowcoloringduality}.
\end{proof}

Using this, we are now able to prove \bothV{Theorem}{Corollary}~\ref{4conn}.

\restatefourconn*
\begin{proof}
Let $G$ be a $4$-edge-connected graph. We claim that for any good drawing $\Gamma$ of $G$, the planarization $G_\text{isc}(\Gamma)$ is $4$-edge-connected: suppose towards a contradiction that for some drawing~$\Gamma$ of $G$ there is an edge\bothV{-}{ }cut $S \subseteq E(G_\text{isc}(\Gamma))$ with $|S| \leq 3$ such that $G_\text{isc}(\Gamma)-S$ is disconnected. Let~$X$,~$Y$ be a partition of $V(G_\text{isc}(\Gamma))$ such that there are no edges between $X$ and $Y$ in $G_\text{isc}(\Gamma)-S$.

\bothV{Suppose, for the moment, that}{Suppose first that} both $X$ and $Y$ contain a normal vertex. Let $S' \subseteq E(G)$ be the set of edges in $G$ whose corresponding paths in $G_\text{isc}(\Gamma)$ connect a normal vertex in $X$ to a normal vertex in~$Y$. Clearly, every such path contains an edge of $S$ and thus we have $|S'| \leq |S| \leq 3$. Since~$G$ is $4$-edge-connected, this means that $G-S'$ is still connected. However, deleting all edges in~$S'$ from the drawing $\Gamma$ disconnects all normal vertices in $X$ from all normal vertices in $Y$. This contradiction shows that either $X$ or $Y$ does not contain normal vertices.

Without loss of generality suppose that $X$ consists only of intersection vertices. Let $x \in X$ be any such intersection vertex and let $e_1=u_1v_1$ and $e_2=u_2v_2$ be the two distinct edges of~$\Gamma$ crossing at $x$. For $i \in \{1,2\}$, let $P_i$ be the path in $G_{\text{isc}}(\Gamma)$ starting in $u_i$ and ending in $v_i$ which corresponds to $e_i$. Since $u_i, v_i$ are normal vertices, they are contained in $Y$. Hence, at least one edge on each of the two subpaths of $P_i$ from $x$ to $u_i$ and from $x$ to $v_i$ must contain an edge in $S$. Hence, for each $i=1,2$, the path $P_i$ contains two distinct edges from $S$, and we have $|S| \ge |S \cap E(P_1)|+|S \cap E(P_2)| \ge 4$, a contradiction.

This shows that the initial assumption was wrong and, hence, any planarization $G_\text{isc}(\Gamma)$ for a good drawing $\Gamma$ of $G$ is $4$-edge-connected. This proves the above claim. By Theorem~\ref{thm:flowcoloringduality}, Proposition~\ref{prop:topssuffice}, and Theorem~\ref{dualgroetzsch}, this implies that $G$ is \bothV{facially}{universally cell} $3$-colorable, concluding the proof.
\end{proof}

\section{An Infinite Family of Counterexamples}\label{sec:infinite}
An initial question of ours was whether maybe also the reverse of Theorem~\ref{thm:3flowsuff} holds true, \bothV{i.e.}{that is}, whether the properties of being \bothV{facially}{universally cell} $3$-colorable and being $3$-flowable are equivalent for all graphs. It turns out that for many graphs, this equivalence does indeed hold (see Section~\ref{sec:mincex}). In contrast, in this section we present a negative answer to this question by constructing an infinite family of \bothV{facially}{universally cell} $3$-colorable graphs which are not $3$-flowable.

Let $m, n \in \mathbb{N}$. Then we denote by $K_{m,n}$ the complete bipartite graph with partite sets of size~$m$ and $n$, and by \bothV{$K_{3,n}^+$}{$K_{3,n-3}^+$} the $n$-vertex-graph obtained from \bothV{$K_{3,n}$}{$K_{3,n-3}$} by adding an edge connecting two vertices in the partite set of size $3$.
{The goal of this section is to prove the following theorem.

\restateKthreens*}

For every \bothV{$n \ge 3$}{$n \ge 7$}, the graph \bothV{$K_{3,n}^+$}{$K_{3,n-3}^+$} is not $3$-flowable, as proved for instance in Proposition~2.5 of{ Li et al.}~\cite{k3n}. \bothV{In the following we show that, as soon as $n \ge 7$, $K_{3,n-3}^+$ forms}{Therefore, it remains to show that, for $n\geq4$, the graph $K_{3,n}^+$ is \bothV{facially}{universally cell} 3-colorable, yielding} a counterexample to the equivalence of \bothV{facially}{universally cell} $3$-colorable and $3$-flowable graphs. As a preparation we need the following (folklore-)fact, a proof of which (in a more general form) can for instance be found in{ the article of Ma\v cajov\'a and Rollov\'a}~\cite{macajova}, compare also Figure~\ref{fig:kmnflows}.

\begin{lemma}\label{lem:compbip}
Let $m, n \in \mathbb{N},m, n \ge 2$. Then $K_{m,n}$ admits a modulo-$3$-orientation. 
\end{lemma}

\begin{figure}[ht]
	\centering
	\includegraphics[scale=1]{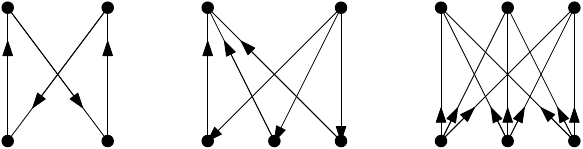}
	\caption{Modulo-$3$-orientations for $K_{m,n}$ with $2 \le n \le m \le 3$.}\label{fig:kmnflows}
\end{figure}

\begin{proposition}
For every \bothV{$n \ge 4$}{$n \ge 7$}, the graph \bothV{$K_{3,n}^+$}{$K_{3,n-3}^+$} is \bothV{facially}{universally cell} $3$-colorable.
\end{proposition}

\begin{proof}
Using Proposition~\ref{prop:topssuffice}, we only need to prove that every good drawing of \bothV{$K_{3,n}^+$}{$K_{3,n-3}^+$} has a \bothV{face-}{cell} $3$-coloring. So consider a given good drawing $\Gamma$ of \bothV{$K_{3,n}^+$}{$K_{3,n-3}^+$}. Let us denote the three vertices in the smaller partite set of \bothV{$K_{3,n} \subseteq K_{3,n}^+$}{$K_{3,n-3} \subseteq K_{3,n-3}^+$} by $x_1, x_2, x_3$, and let \bothV{$x_1x_2 \in E(K_{3,n}^+)$}{$x_1x_2 \in E(K_{3,n-3}^+)$} \bothV{be the new edge added}{be the edge added} to \bothV{$K_{3,n}$}{$K_{3,n-3}$}. Let $B$ denote the second partite set of size $n-3$. Since \bothV{$K_{3,n}$}{$K_{3,n-3}$} is a non-planar graph, {at least }two non-adjacent edges with endpoints in $A\coloneqq\{x_1,x_2,x_3\}$ and $B$ must cross in the drawing $\Gamma$. Let $e_1=u_1v_1 \neq e_2=u_2v_2$ be two such crossing edges, where $u_1 \neq u_2 \in A, v_1 \neq v_2 \in B$. Let $H$ be the auxiliary abstract graph obtained from \bothV{$K_{3,n}^+$}{$K_{3,n-3}^+$} by deleting the edges $e_1$ and $e_2$, adding a new vertex $u \notin A \cup B$ to the vertex set and making $u$ adjacent to the original endpoints $u_1,v_1,u_2,v_2$ of $e_1$ and $e_2$. In the following, our goal is to show that $H$ admits a modulo-$3$-orientation. 
To do so, we distinguish between two cases depending on the adjacency of $u_1$ and $u_2$.

\textbf{Case 1.} $u_1$ and $u_2$ are adjacent. Since $x_1x_2$ is the only edge in \bothV{$K_{3,n}^+$}{$K_{3,n-3}^+$} between vertices in~$A$, \bothV{possibly after relabelling we may assume}{we may assume (if necessary by relabeling)} that $u_i=x_i, i=1,2$. We will now construct two modulo-$3$-orientations $D_1$ and $D_2$ of disjoint subgraphs $H_1$ and $H_2$ of $H$ which partition the edges of $H$. {The graph }$H_1$ is the induced subgraph $H[\{x_1,x_2,x_3,u,v_1,v_2\}]$, and \mbox{$H_2\coloneqq H-E(H_1)$} is isomorphic to the disjoint union of \bothV{$K_{3,n-2}$}{$K_{3,n-5}$} with the three isolated vertices $u, v_1, v_2$. The orientation $D_1$ of $H_1$ is depicted in \bothV{Figure}{Fig.}~\ref{caseorientations}, while $D_2$ is chosen as a modulo-$3$-orientation of~\bothV{$K_{3,n-2}$}{$K_{3,n-5}$}, whose existence is guaranteed by Lemma~\ref{lem:compbip} (here we use that \bothV{$n-2 \ge 4-2=2$}{$n-5 \ge 7-5=2$}). It is apparent that the arc-disjoint union of two modulo-$3$-orientations still defines a modulo-$3$-orientation. Hence, $D\coloneqq D_1 \cup D_2$ is a modulo-$3$-orientation of $H=H_1 \cup H_2$, and this concludes Case 1.

\begin{figure}[ht]
\centering
\includegraphics[scale=0.7]{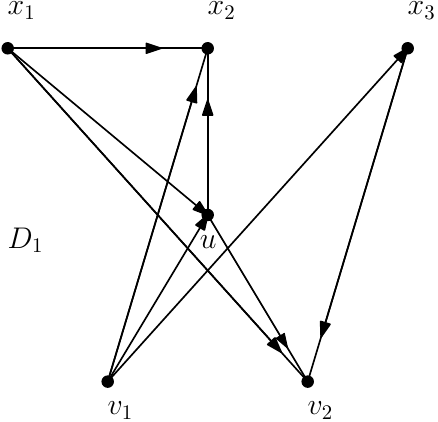} \hspace{50pt} \includegraphics[scale=0.7]{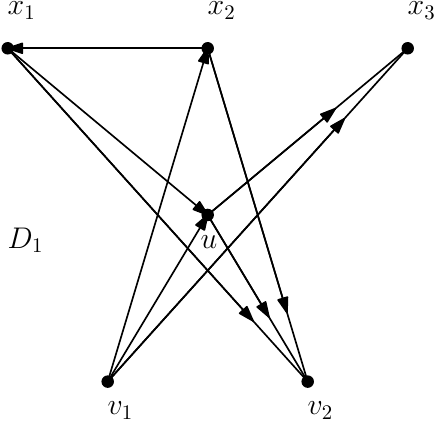}
\caption{The orientation $D_1$ of $H_1$, on the left for Case 1 and on the right for Case 2.}\label{caseorientations}
\end{figure}

\textbf{Case 2.} $u_1$ and $u_2$ are non-adjacent. \bothV{Possibly after relabelling, we may assume}{We may assume (if necessary by relabeling)} that $u_1=x_1$ and $u_2=x_3$. Let us again consider a partition of $H$ into two edge-disjoint subgraphs $H_1$ and $H_2$, where $H_1$ is the induced subgraph $H[\{x_1,x_2,x_3,u,v_1,v_2\}]$, and $H_2\coloneqq H-E(H_1)$ is the disjoint union of a \bothV{$K_{3,n-2}$}{$K_{3,n-5}$} with the three isolated vertices $u,v_1$ and $v_2$. Let $D_1$ be the orientation of $H_1$ depicted in {Fig.}~\ref{caseorientations}. We have $\text{exc}_{D_1}(v_1) \equiv \text{exc}_{D_1}(v_2) \equiv \text{exc}_{D_1}(u) \equiv 0 \text{ (mod }3)$, but $\text{exc}_{D_1}(x_i) \equiv 1 \text{ (mod }3)$ for $i=1,2,3$.

\begin{figure}[ht]
\centering
\includegraphics[scale=0.6]{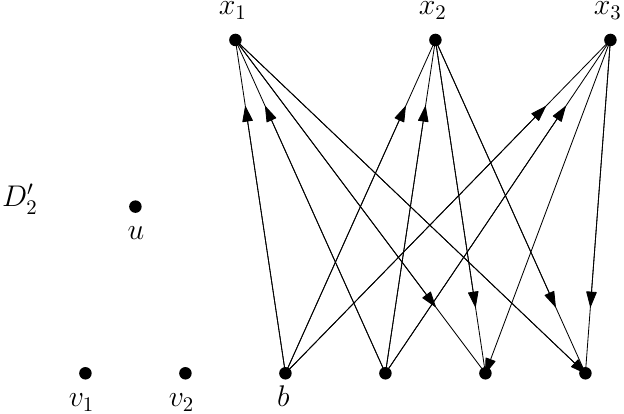} \hspace{10pt} \includegraphics[scale=0.6]{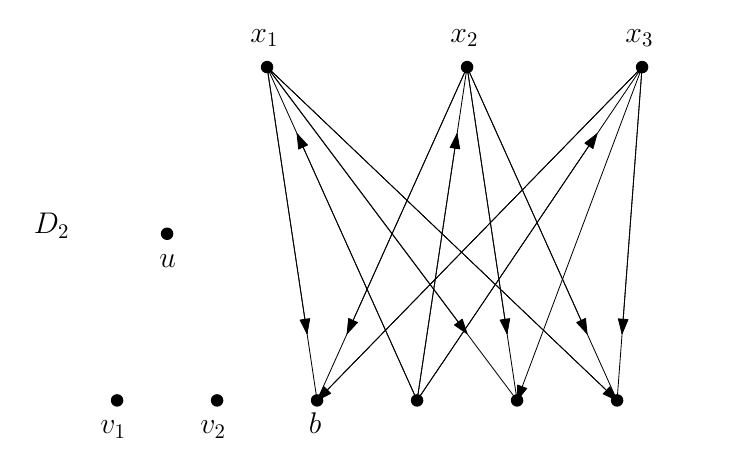}
\caption{A modulo $3$-orientation $D_2'$ of $H_2$ and the orientation $D_2$ obtained from it, exemplified in the case $n=9$.}\label{D2s}
\end{figure}

Further let $D_2$ be an orientation of $H_2$ such that $\text{exc}_{D_2}(v) \equiv 0 \text{ (mod }3)$ for every $v \in B$ but $\text{exc}_{D_2}(x_i) \equiv 2 \text{ (mod }3)$ for every $i \in \{1,2,3\}$.
Such an orientation exists, which can be seen as follows {(compare the example in Fig.~\ref{D2s})}: Take a modulo-$3$-orientation $D_2'$ of $H_2$ (which exists as guaranteed by Lemma~\ref{lem:compbip} since \bothV{$n-2 \ge 2$}{$n-5 \ge 2$}). Let $b \in B \setminus \{v_1,v_2\}$ be chosen arbitrarily. Since $b$ has degree $3$ in $H_2$ with neighbors $x_1,x_2,x_3$, it is either a source or a sink in $D_2'$. Hence, possibly after reversing all arcs in $D_2'$ we may assume $(b,x_i) \in A(D_2'), i=1,2,3$. We can now define $D_2$ as the orientation of $H_2$ obtained from \bothV{$D_1'$}{$D_2'$} by reversing the three arcs incident to $b$. Still, the excess of every vertex in $B$ with respect to $D_2$ is divisible by $3$, while every $x_i$ loses an in-arc and gains an out-arc, which means that it has excess $2$ modulo $3$ in $D_2$. This shows that $D_2$ has the required properties.

We now finally define an orientation $D$ of $H$ as the disjoint union of $D_1$ and $D_2$. It is now clear that we have $\text{exc}_D(x) \equiv \text{exc}_{D_1}(x)+\text{exc}_{D_2}(x) \equiv 0+0=0 \text{ (mod }3)$ for every vertex $v \in V(H) \setminus A$ and $\text{exc}_D(x_i)=\text{exc}_{D_1}(x_i)+\text{exc}_{D_2}(x_i) \equiv 1+2 \equiv 0 \text{ (mod }3)$ for $i=1,2,3$. This proves that $D$ is a modulo-$3$-orientation of $H$ and, thereby, concludes Case 2.

It remains to be seen how to obtain a \bothV{face-}{cell} $3$-coloring of $\Gamma$ from the modulo-$3$-orientation $D$ of $H$. Indeed, since $e_1$ and $e_2$ cross in the drawing $\Gamma$, by placing the vertex $u$ at (an arbitrarily chosen) crossing of $e_1$ and $e_2$, we obtain a drawing $\Gamma'$ of $H$, whose induced cell decomposition of $\mathbb{R}^2$ is combinatorially equivalent to the one of $\Gamma$. Since $H$ is $3$-flowable{ by Lemma~\ref{mod3or}}, we can apply Theorem~\ref{thm:3flowsuff} to $H$ in order to conclude that $\Gamma'$, and hence $\Gamma$ as well, admits a \bothV{face-}{cell }$3$-coloring. Since $\Gamma$ was chosen as an arbitrary {good }drawing of \bothV{$K_{3,n}^+$}{$K_{3,n-3}^+$}, this shows that \bothV{$K_{3,n}^+$}{$K_{3,n-3}^+$} is \bothV{facially}{universally cell} $3$-colorable, as claimed.
\end{proof}

\section{Towards Characterizing \bothV{Facially}{Universal\noldV{ly} Cell} 3-Colorab\nbothV{le Graphs}{ility}}\label{sec:mincex}
In this section, we study \bothV{facially}{universally cell} $3$-colorable graphs in more detail. We first prepare the proofs of Theorem~\ref{scubic} and Theorem~\ref{K33free}, by Proposition~\ref{subcontclosed}, which is crucial to both proofs. We further discuss Conjecture~\ref{main} and derive some properties a smallest counterexample to this conjecture must have (Theorem~\ref{thm:mincex}).
The following {proposition }shows that the class of \bothV{facially}{universally cell} $3$-colorable graphs is closed with respect to taking subcontractions. 

\begin{proposition}\label{subcontclosed}
Let $G$ be a \bothV{facially}{universally cell} $3$-colorable graph, and let $X \subseteq V(G)$. Then the graph $G/X$ is \bothV{facially}{universally cell} $3$-colorable as well.
\end{proposition}

\begin{proof}
Since the graph $G/X$ can be obtained from $G$ by repeatedly identifying pairs of vertices, it is sufficient to prove the claim in the case $|X|=2$. So let $X=\{u,v\}$ for some $u, v \in V(G)$ with $u \neq v$. Let $\Gamma_{uv}$ be any given drawing of the graph $G/\{u,v\}$. Let $x_{uv}$ denote the vertex in~$G/\{u,v\}$ obtained by identifying $u$ and $v$. Let $p \in \mathbb{R}^2$ be the position of $x_{uv}$ in the drawing $\Gamma_{uv}$. Let $\varepsilon>0$ be small enough such that the closed \bothV{ball}{\nbothV{disk}{ball}} $B_\varepsilon(p)$ with radius $\varepsilon$ around $x_{uv}$ contains no other vertices or crossing points, and such that every edge incident to $x_{uv}$ in the drawing $\Gamma_{uv}$ intersects the boundary of $B_\varepsilon(p)$ exactly once. Let $e_1,e_2,\ldots,e_r$ denote the incident edges of $x_{uv}$ in $G/\{u,v\}$ which arise from the edges in $G$ incident to $u$, while $e_{r+1},\ldots,e_\ell$ denote the edges arising from the edges incident to $v$ in $G$. Let $p_i$, $i=1,\ldots,\ell$, denote the intersection point of the curve $\gamma(e_i) \subseteq \mathbb{R}^2$ representing $e_i$ with the boundary of $B_\varepsilon(p)$. 
We now locally modify the drawing $\Gamma_{uv}$ within $B_{\varepsilon}(p)$ to obtain a drawing $\Gamma$ of $G$ as follows (see also {Fig.}~\ref{redrawing}): first, we delete all \bothV{features of}{edge\noldV{-}parts of} the drawing $\Gamma_{uv}$ contained in $B_\varepsilon(p)$ {and $x_{uv}$}. Next we place the vertices $u$ and $v$ at two distinct positions in the interior of $B_\varepsilon(p)$. We connect (with straight-line segments) $u$ to the points $p_1,\ldots,p_r$ and $v$ to the points $p_{r+1},\ldots,p_\ell$ possibly introducing crossings and draw all edges between $u$ and $v$ (including possible loops) within $B_\varepsilon(p)$ introducing no further crossings. We now join $\gamma(e_i) \setminus B_\varepsilon(p)$ with the straight-line segment from $p_i$ to $u$ or $v$ for $i=1,\ldots,\ell$ and thereby obtain the drawing $\Gamma$ of $G$. Since $G$ is \bothV{facially}{universally cell} $3$-colorable, $\Gamma$ admits a \bothV{face-}{cell }$3$-coloring{, call it $c$}.
%, which by Theorem~\ref{thm:flowcoloringduality} yields a nowhere-zero $3$-flow $(D,f)$ of $G_\text{isc}(\Gamma)$. Let $I \subseteq V(G_\text{isc}(\Gamma))$ contain $u, v$ as well as all the crossing vertices corresponding to crossings contained in the ball $B_\varepsilon(p)$. From the above construction of $\Gamma$, it follows that $G_\text{isc}(\Gamma_{uv})$ is isomorphic to $G_\text{isc}(\Gamma)/I$. By restricting the flow $(D,f)$ to the edges that are not contracted in $G_\text{isc}(\Gamma)/I$, we see that this graph admits a nowhere-zero $3$-flow as well. Hence, another application of Theorem~\ref{thm:flowcoloringduality} yields that $G_\text{isc}(\Gamma_{uv})$ and thus the drawing $\Gamma_{uv}$ admits a face-$3$-coloring. 
Now we can \bothV{face-$3$-color $\Gamma_{uv}$ easily}{color the cells of $\Gamma_{uv}$ with three colors, as follows}. \bothV{Outside $B_\varepsilon(p)$, the coloring is the same as for $\Gamma$}{If a cell $f$ of~$\Gamma_{uv}$ is disjoint from $B_\varepsilon(p)$, then it also exists as a cell in $\Gamma$ and we assign to it the same color~$c(f)$ as in $\Gamma$}.
 \bothV{Since every face of $\Gamma_{uv}$ has a part outside $B_\varepsilon(p)$ and all adjacencies inside $B_\varepsilon(p)$ are realized outside of it just next to it, this coloring is a valid face-3-coloring of $\Gamma_{uv}$}{If a cell $f$ of $\Gamma_{uv}$ intersects the ball $B_\varepsilon(p)$, then by choice of $\varepsilon$, $f\setminus B_\varepsilon(p) \neq \emptyset$, since the edges in $\Gamma_{uv}$ incident to $x_{uv}$ do not intersect inside the ball $B_\varepsilon(p)$.
  Since $\Gamma$ and $\Gamma_{uv}$ are the same outside of $B_\varepsilon(p)$, there exists a unique cell $f'$ of $\Gamma$ such that $f\setminus B_\varepsilon(p) \subseteq f'$, and we assign to $f$ the color $c(f')$ that was assigned to $f'$ in the cell $3$-coloring $c$ of $\Gamma$. Note that whenever two cells $f_1,f_2$ in $\Gamma_{uv}$ intersecting $B_\varepsilon(p)$ are adjacent, then their boundaries contain a common edge incident to $x_{uv}$, and therefore also the corresponding cells $f_1',f_2'$ in $\Gamma$ touch. 
  Hence, the so-defined $3$-coloring of the cells of $\Gamma_{uv}$ is a proper cell $3$-coloring}. Since $\Gamma_{uv}$ was chosen as an arbitrary drawing of $G/\{u,v\}$, this proves that $G/\{u,v\}$ is \bothV{facially}{universally cell} $3$-colorable.
\end{proof}

\begin{figure}[ht]
\centering
\includegraphics[scale=0.7]{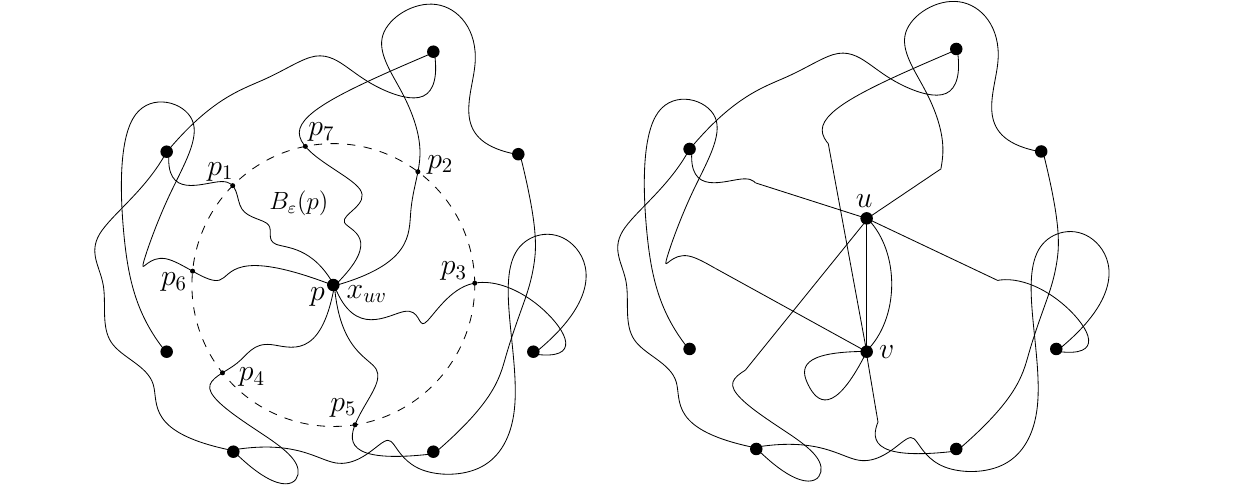}
\caption{Illustration of the construction of the drawing $\Gamma$ (right) from the drawing $\Gamma_{uv}$ (left).}\label{redrawing}
\end{figure}

{For later use, we also include the following statement, which belongs to the folklore of nowhere-zero-flow theory. It states that also the $k$-flowable graphs are closed under subcontractions for every fixed $k \in \mathbb{N}$. 

\begin{lemma}\label{lemma:subcontflows}
If $G$ is a $k$-flowable graph, then also $G/X$ is $k$-flowable for any $X\subseteq V(G)$. In particular, every subcontraction of a $k$-flowable graph is $k$-flowable.
\end{lemma}
\begin{proof}
For a $k$-flowable graph $G$, suppose that $(D,f)$ describes a nowhere-zero $k$-flow of $D$, where $D$ is an orientation of $G$ and \mbox{$f\colon A(D) \rightarrow \mathbb{Z}_k\setminus\{0\}$} is a group-valued flow on $D$. Let $v_X$ be the vertex such that \mbox{$V(G/X)=(V(G)\setminus X) \cup \{v_X\}$}. Next define a pair $(D',f')$ consisting of an orientation of $G/X$ and a mapping $f'\colon A(D') \rightarrow \mathbb{Z}_k\setminus \{0\}$ as follows: Every edge of $G/X$ not incident to $v_X$ is oriented in the same way as its corresponding edge in $D$. Further, for every edge $e=wv_X \in E(G/X)$ coming from an edge $wv \in E(G)$ with $w \in V(G)\setminus X, v \in X$, we orient~$e$ from $w$ towards $v_X$ if $(w,v) \in A(D)$, and from $v_X$ to $w$ if $(v,w) \in A(D)$. Finally, we define $f'\colon A(D') \rightarrow \mathbb{Z}_k\setminus \{0\}$ by assigning the value $f'(e')\coloneqq f(e)$ to every edge $e' \in A(D')$ with corresponding edge $e$ in $D$. We now claim that $(D',f')$ is a nowhere-zero $k$-flow for $G/X$. Indeed, it is readily verified using the above definitions that for every fixed $v \in V(D') \setminus \{v_X\}$, we have
$$\sum_{e'=(w',v) \in A(D')}{f'(e')}=\sum_{e=(w,v) \in A(D)}{f(e)}=\sum_{e=(v,w) \in A(D)}{f(e)}=\sum_{e'=(v,w') \in A(D')}{f'(e')},$$ as required by Kirchhoffs' law. For the vertex $v_X$, we can observe that
$$\sum_{e'=(w',v_X) \in A(D')}{f'(e')}=\sum_{e \in A(D)
\cap ((V(D)\setminus X) \times X)}{f(e)}$$
and symmetrically
$$\sum_{e'=(v_X,w') \in A(D')}{f'(e')}=\sum_{e \in A(D)
\cap (X \times (V(D)\setminus X))}{f(e)}.$$
Hence, to verify Kirchhoff's law at $v_X$, it suffices to show that the two right-hand sides in the above equations coincides. Using that $f$ satisfies Kirchhoff's law at all $v \in X$, we can obtain this equality as follows:
$$\sum_{e \in A(D)
\cap ((V(D)\setminus X) \times X)}{f(e)}=\left(\sum_{v \in X}\sum_{e=(w,v) \in A(D)}{f(e)}\right)-\sum_{e \in A(D) \cap (X \times X)}{f(e)}$$
$$=\left(\sum_{v \in X}\sum_{e=(v,w) \in A(D)}{f(e)}\right)-\sum_{e \in A(D) \cap (X \times X)}{f(e)}=\sum_{e \in A(D)
\cap (X \times (V(D)\setminus X))}{f(e)}.$$
\end{proof}
}

With \oldV{the }Proposition~\ref{subcontclosed} {and Lemma~\ref{lemma:subcontflows}} as {tools} in hand, we are ready for the proofs of Theorem~\ref{scubic} and Theorem~\ref{K33free}. {Before showing the proof of Theorem~\ref{scubic}, we show a statement of independent interest in Proposition~\ref{prop:oddwheel}, which will then be used in the proof. An \emph{odd wheel} is a simple graph obtained from an odd cycle by adding a dominating vertex.
\begin{proposition}\label{prop:oddwheel}
All odd wheels are not universally cell $3$-colorable.
\end{proposition}
\begin{proof}
Let $G$ be an odd wheel and consider the canonical planar drawing of $G$ in which the dominating vertex is enclosed by the cycle spanned by the remainder of the vertices. This drawing does not admit a cell $3$-coloring: Its planar dual graph is also an odd wheel (on the same number of vertices) and hence has chromatic number $4$, since every odd cycle has chromatic number $3$ and since the addition of a dominating vertex to a graph raises its chromatic number by $1$.
\end{proof}
}

\restatecubic*
\begin{proof}
We prove the theorem by showing that, for a graph $G$ with maximum degree at most $3$, the following three statements are equivalent. \oldV{An \emph{odd wheel} is a simple graph obtained from an odd cycle by adding a dominating vertex.}
\begin{enumerate}[(i)]
	\item $G$ is \bothV{facially}{universally cell} $3$-colorable.
	\item $G$ is bridgeless and \bothV{has no odd wheel as a subcontraction}{has no subcontraction isomorphic to an odd wheel}.
	\item $G$ is $3$-flowable.
\end{enumerate}

\textbf{(iii) $\Rightarrow$ (i)} This follows from Theorem~\ref{thm:3flowsuff}.

\textbf{(i) $\Rightarrow$ (ii)} Let $G$ be a graph of maximum degree $3$ which is \bothV{facially}{universally cell} $3$-colorable. $G$ must be bridgeless, since otherwise if $e \in E(G)$ was a bridge of $G$, we could draw $G$ in a way such that $e$ is adjacent on both sides to the outer \bothV{face}{cell}. This, however, means that the outer \bothV{face}{cell} \bothV{has a self-touching}{is adjacent to itself}, contradicting the fact that this drawing has a \bothV{face-}{cell }$3$-coloring. Further, by Proposition~\ref{subcontclosed} every subcontraction of $G$ is \bothV{facially}{universally cell} $3$-colorable. However, \bothV{all odd wheels are self-dual planar graphs of chromatic number $4$, and hence they are not facially $3$-colorable}{by Proposition~\ref{prop:oddwheel} all odd wheels are not universally cell $3$-colorable}. This shows that no subcontraction of $G$ is isomorphic to an odd wheel.

\textbf{(ii) $\Rightarrow$ (iii)} Suppose towards a contradiction that this implication does not hold, \bothV{i.e.}{that is}, there exists a bridgeless graph~$G$ of maximum degree $3$ without an odd wheel as a subcontraction, but~$G$ is not $3$-flowable{, hence by Lemma~\ref{mod3or}, $G$ has no modulo-$3$-orientation}. Let us choose $G$ such that it minimizes $|V(G)|$ with respect to these conditions. Clearly, we must have $|V(G)| \ge 3$.

We claim that $G$ must be connected. Indeed, suppose towards a contradiction that $G$ is the disjoint union of two subgraphs $G_1, G_2$. Clearly, both $G_1$ and $G_2$ are bridgeless graphs of maximum degree at most $3$ which are isomorphic to subcontractions of $G$ (indeed, for $i \in \{1,2\}$ the graph $G/X$ is isomorphic to $G_i$, where $X\coloneqq V(G_{3-i}) \cup \{v\}$ for some vertex $v \in V(G_i)$). This implies that neither $G_1$ nor $G_2$ have an odd wheel as a subcontraction, since such an odd wheel would also be a subcontraction of $G$. Since $|V(G_i)|<|V(G)|$ for $i=1,2$, the minimality of $G$ implies that $G_1$ and $G_2$ admit modulo-$3$-orientations $D_1$ and $D_2$. However, now $D_1 \cup D_2$ defines a modulo-$3$-orientation of $G$, which contradicts our initial assumptions on $G$. Hence, $G$ is connected.

\begin{figure}[ht]
	\centering
	\includegraphics[scale=0.7]{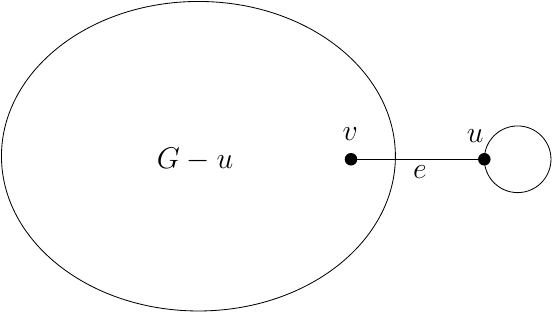}
	\caption{Illustration of the case in the proof when $G$ contains a loop at the vertex $u$.}\label{case1max3}
\end{figure}

\begin{figure}[ht]
	\centering
	\includegraphics[scale=0.7]{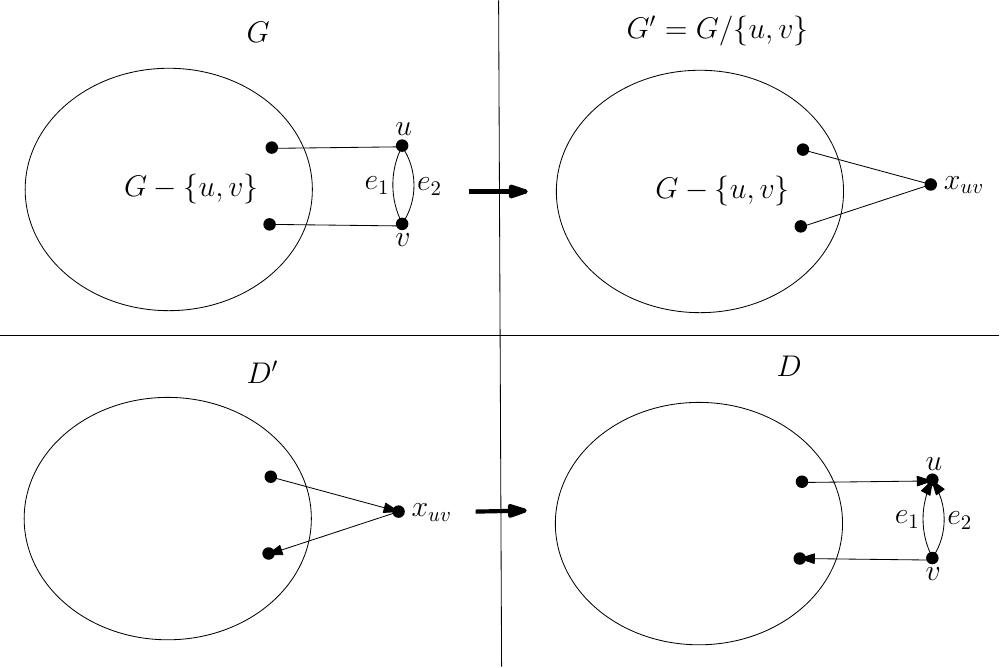}
	\caption{Illustration of the case in the proof when $G$ contains the parallel edges $e_1, e_2$. On the top, the construction of $G'$ from $G$ is illustrated, and at the bottom, the construction of $D$ from $D'$ is illustrated.}\label{case2max3}
\end{figure}

Next we claim that $G$ must be a simple graph. If not, then $G$ contains either a loop or two parallel edges between a pair of vertices. In the first case let $u \in V(G)$ be incident to a loop{ (see Fig.~\ref{case1max3})}. Since $|V(G)| \ge 2$, there must be another edge $e$ connecting $u$ with a neighbor $v \in V(G)\setminus \{u\}$. Since $u$ has degree at most $3$, the edge $e$ must be a bridge, contradicting the fact that $G$ is bridgeless. In the second case let $u \neq v \in V(G)$ be a pair of vertices such that there are distinct edges $e_1 \neq e_2 \in E(G)$ with endpoints $u$ and $v${ (see Fig.~\ref{case2max3})}. Since $G$ is 2-edge-connected, has maximum degree 3 and $|V(G)|\ge 3$, there are two edges connecting $u$ and $v$ to other vertices. 
In particular, there is no third parallel edge from $u$ to $v${.} Let us now consider the graph $G'\coloneqq G/\{u,v\}$. {Note that by definition, $G'$ is a subcontraction of $G$.} Since~$G$ is bridgeless, so is $G'$, and since $G'$ is a subcontraction of $G$, it does not have an odd wheel as subcontraction. Further, $G'$ has no vertex of degree more than $3$: identifying $u$ and $v$ leaves the degree of all vertices in $V(G)\setminus \{u,v\}$ unchanged, while the identification vertex $x_{uv} \in V(G')$ must have degree at most $2$. Since $|V(G')|<|V(G)|$, the minimality of $G$ implies that $G'$ has a modulo-$3$-orientation $D'$. By expanding $x_{uv}$ into $u$ and $v$ and keeping the orientations of edges in $D'$, we obtain an orientation $D^\ast$ of $G-\{e_1,e_2\}$ such that $\text{exc}_{D^\ast}(x)=\text{exc}_D(x) \equiv 0 \text{ (mod }3)$ for all $x \in V(G)\setminus \{u,v\}$, and $\text{exc}_{D^\ast}(u)+\text{exc}_{D^\ast}(v)=\text{exc}_{D'}(x_{uv}) \equiv 0 \text{ (mod }3)$. Let $D$ be defined as the orientation of $G$ obtained from $D^\ast$ by orienting the two parallel edges $e_1, e_2$ as follows: if $\text{exc}_{D^\ast}(u) \equiv \text{exc}_{D^\ast}(v) \equiv 0 \text{ (mod }3)$, then we orient $e_1$ from $u$ towards $v$ and $e_2$ from $v$ towards~$u$. If $\text{exc}_{D^\ast}(u) \equiv 1, \text{exc}_{D^\ast}(v) \equiv 2 \text{ (mod }3)$, then we orient both $e_1$ and $e_2$ from $u$ towards $v$. Finally, if $\text{exc}_{D^\ast}(u) \equiv 2, \text{exc}_{D^\ast}(v) \equiv 1 \text{ (mod }3)$, then we orient both $e_1$ and $e_2$ from $v$ towards $u$. In each case, the obtained orientation $D$ of $G$ defines a modulo-$3$-orientation of $G$, however $G$ is not $3$-flowable. This contradiction shows that our assumption that $G$ is not simple was wrong. Hence, we have established that $G$ is a simple graph of maximum degree at most $3$. 

We claim that in fact, $G$ must be a cubic graph. Suppose towards a contradiction that there was a vertex $v \in V(G)$ with only two incident edges $f_1$, $f_2$. Then the subcontraction $G/f_1$ of~$G$ is bridgeless and has maximum degree at most $3$. Hence, by the minimality of $G$ there exists a modulo-$3$-orientation $D'$ of $G/f_1$. Expanding this orientation to $G$ and orienting $f_1$ in such a way that $f_1$ and $f_2$ form a directed path of length two, we find a modulo-$3$-orientation $D$ of $G$, again contradicting that $G$ is not $3$-flowable. This shows that $G$ is cubic.

We finally claim that $G$ is bipartite. Suppose not, then there exists an induced cycle $C$ in~$G$ of odd length. Since every vertex in $V(C)$ has a unique third neighbor in $V(G)\setminus V(C)$, the subcontraction $G/(V(G) \setminus V(C))$ of $G$ is isomorphic to an odd wheel. This is a contradiction to our initial assumption that $G$ does not have an odd wheel as a subcontraction.

We summarize: $G$ is a simple, bipartite, and cubic graph. Let $V(G)=A \cup B$ be a bipartition of $G$. Consider the orientation of $G$ in which all edges are oriented from $A$ towards $B$. In this orientation, every vertex in $A$ has excess $3$, while every vertex in $B$ has excess $-3$. Hence, this is a modulo-$3$-orientation, contradicting that $G$ is not $3$-flowable. Thus, our initial assumption, namely that there exists a bridgeless graph $G$ of maximum degree at most $3$ without odd wheel as a subcontraction that is not $3$-flowable, was wrong. This proves (ii) $\Rightarrow$ (iii).
\end{proof}

The following corollary follows directly from the proof of Theorem \ref{scubic}:

\begin{corollary}
A cubic graph is \bothV{facially}{universally cell} 3-colorable if and only if it is bipartite.\label{cubic}
\end{corollary}

Our next goal is to prove Theorem~\ref{K33free}. {Since the proof is a bit technical in parts, before jumping right into its details, let us give a general overview of the strategy we use in the following.

Recall that our goal is to show that any graph $G$ with no $K_{3,3}$-minor is universally cell $3$-colorable if and only if it is $3$-flowable. By Theorem~\ref{thm:3flowsuff}, we already know that the existence of a nowhere zero $3$-flow is sufficient for universal cell $3$-colorability (even for arbitrary graphs), so the essential part of the proof is to show that if a graph $G$ is $K_{3,3}$-minor-free and universally cell $3$-colorable, then it admits a nowhere zero $3$-flow. 
Since this implication does not hold true for general graphs by the examples from Section~\ref{sec:infinite}, our proof needs to make use of additional structure that $K_{3,3}$-minor free graph{s} have. Luckily, a precise structure theorem for the class of $K_{3,3}$-minor-free graphs was established in a classical result of Wagner~\cite{thomas,wagner}. Basically, the structure theorem tells us that a graph has no $K_{3,3}$-minor if and only if it can be built from a set of planar graphs, as well as a set of copies of $K_5$, by glueing these graphs together along small separations (of order two). More precisely,} 
given a pair $G^1, G^2$ of simple graphs such that $V(G^1) \cap V(G^2)$ induces \bothV{cliques of order $i$ in both}{a clique of order $i$ in each of}~$G^1$ and $G^2$, and such that $|V(G^1)|,|V(G^2)|>i$, the simple graph $G$ with $V(G)=V(G^1) \cup V(G^2)$ and $E(G)=E(G^1) \cup E(G^2)$ is called the \emph{proper $i$-sum} of~$G^1$ and $G^2$. A graph obtained from $G$ by deleting a subset (possibly all or none) of the edges in $E(G^1) \cap E(G^2)$ is said to be an \emph{$i$-sum} of $G^1$ and $G^2$.

\begin{theorem}[Wagner \cite{thomas,wagner}]\label{wagner}
A simple graph is $K_{3,3}$-minor-free if and only if it can be obtained from simple planar graphs and graphs isomorphic to $K_5$ by means of repeated $i$-sums, where $i \in \{0,1,2\}$.
\end{theorem}

{Looking at Theorem~\ref{wagner}, it is apparent that it should ease our task of proving that a universally cell $3$-colorable graph $G$ with no $K_{3,3}$-minor is $3$-flowable: For \emph{both} building blocks of $K_{3,3}$-minor free graphs, namely for planar graphs and for copies of $K_5$, we already know that the equivalence of universal cell $3$-colorability and $3$-flowability holds (compare Theorems~\ref{thm:flowcoloringduality} and~\ref{thm:3flowsuff}, and note that~$K_5$ is $3$-flowable). Hence, the only obstacle in the proof will be to deal with larger graphs glued by copies of planar graphs and $K_5$, and to show that they do not bring additional harm. The main tool to handle this task will be Proposition~\ref{subcontclosed}: In the proof, we will consider (by way of contradiction) a smallest $K_{3,3}$-minor-free graph which is universally cell $3$-colorable, but not $3$-flowable. Then, Proposition~\ref{subcontclosed}, together with the minimality of $G$, assures that whenever we consider a non-empty subset of edges $E \subseteq E(G)$, the graph $G/E$ will still be universally cell $3$-colorable, and clearly, still $K_{3,3}$-minor free, and hence admit a nowhere zero $3$-flow. This nice property will then be used to show that by combining different nowhere zero $3$-flows of different subcontractions of $G$, we can actually find a nowhere zero $3$-flow of $G$ itself (which then concludes the proof). In order for this last step of the proof to work out, it will be essential to contract particularly selected edge\noldV{-}subsets in a controlled way, such that the partial flows obtained on the contracted graphs fit together nicely on the whole graph. For this, we will need the following consequence of Theorem~\ref{wagner}, which describes the structure of a $K_{3,3}$-minor free graph $G$ that is non-planar in a global way (more explicit than the glueing description in Theorem~\ref{wagner}). }

\oldV{The following consequence of Theorem~\ref{wagner} will be used in the proof of Theorem~\ref{K33free}.}

\begin{lemma}\label{K5structure}
Let $G$ be a $2$-vertex-connected $K_{3,3}$-minor-free graph. If $G$ is not planar, then there exist connected subgraphs $G_{i,j}, 1 \leq i<j \leq 5$ of $G$ and distinct vertices $v_1,\ldots,v_5 \in V(G)$ {(see Fig.~\ref{fig:K5struct} for illustration)} such that 
\begin{enumerate}
	\item $E(G)=\bigcup_{i,j}{E(G_{i,j})}$,
	\item $V(G_{i,j}) \cap \{v_1,\ldots,v_5\}=\{v_i,v_j\}$ for all $1 \leq i<j \leq 5$, and
	\item the sets $V(G_{i,j})\setminus \{v_i,v_j\}, 1 \leq i<j \leq 5,$ are pairwise disjoint.
\end{enumerate}
\end{lemma}

\begin{figure}[ht]
\centering
\includegraphics[scale=0.8]{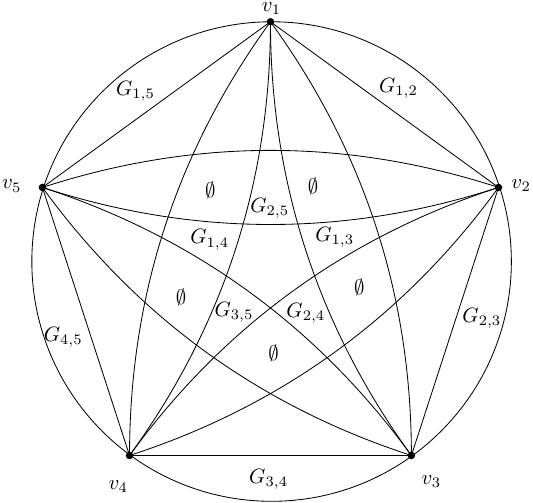}
\caption{Illustration of Lemma \ref{K5structure}: the symbol $\emptyset$ illustrates that while there might be crossings, there are no common vertices of the corresponding subgraphs $G_{i,j}$.}
\label{fig:K5struct}
\end{figure}
\begin{proof}
Let us first prove the statement for simple graphs. Suppose towards a contradiction the claim was false for simple graphs and consider a counterexample $G$ of minimum order. Since $G$ is non-planar and not isomorphic to $K_5$ (otherwise the claim of the lemma holds trivially true), it follows from Theorem~\ref{wagner} that $G$ can be written as the $i$-sum ($i\le 2$) of two simple $K_{3,3}$-minor free graphs $G^1, G^2$. Since $V(G^1) \cap V(G^2)$ forms a vertex-separator of size $i$ in $G$, and since $G$ is assumed to be $2$-vertex-connected, we furthermore know that $i=2${, since $V(G^1)\setminus V(G^2)$ and $V(G^2)\setminus V(G^1)$ are non-empty by definition of an~\mbox{$i$-sum}}. Thus, $E(G^1) \cap E(G^2)$ consists of exactly one edge $e=u_1u_2$. Since $2$-sums of planar graphs are planar, we know that at least one of $G^1, G^2$, say $G^1$, is non-planar. We claim that~$G^1$ has to be $2$-vertex-connected: suppose towards a contradiction there was a vertex $v \in V(G^1)$ such that~\mbox{$G^1-v$} is disconnected. 
Let $X$ be the vertex set of the connected component of~\mbox{$G^1-v$} containing $\{u_1,u_2\} \setminus \{v\}$. Now in $G-v$, no vertex in $V(G^1-v)\setminus X$ has an edge to~$u_1$ or $u_2$, nor to a vertex in $V(G^2) \setminus \{u_1,u_2\}$. Hence, $G-v$ is disconnected, contradicting the assumed $2$-connectivity of~$G$. We therefore conclude that $G^1$ is a non-planar, $2$-connected, simple $K_{3,3}$-minor-free graph, and so by the assumed minimality of $G$, we conclude that there exist connected subgraphs~$(G^1)_{i,j}, 1 \le i<j \le 5$ of $G^1$ and vertices $v_1,\ldots,v_5 \in V(G^1)$ satisfying the properties~\mbox{1--3}\oldV{ claimed in the Lemma}. Suppose without loss of generality that $(G^1)_{1,2}$ is the subgraph containing the edge $e$. We now define an edge-decomposition of $G$ into subgraphs $G_{i,j}, 1 \le i <j \le 5$ as follows: for every pair $\{i,j\}\neq \{1,2\}$, we let $G_{i,j}\coloneqq (G^1)_{i,j}$, while we define $G_{1,2}$ to be obtained from the proper $2$-sum of $(G^1)_{1,2}$ with $G_2$ along $e=u_1u_2$ by removing the edge $e$ if and only if $e \notin E(G)$. Note that the graph $G_{1,2}$ is connected: since $G$ is $2$-connected, $G^2-e$ must be connected, and so each vertex in $(G^1)_{1,2}$ lies in the same connected component as $u_1$ and $u_2$. Let us now show that the graphs~$G_{i,j}, 1\le i<j \le 5$ together with the vertices $v_1,\ldots,v_5$ satisfy the properties~1--3 claimed in the lemma. Property 1 follows directly from the definition of the graphs, and Properties 2 and 3 follow since $\{v_1,\ldots,v_5\}$ and $V(G_{i,j}), \{i,j\} \neq \{1,2\}$ are disjoint from $V(G^2) \setminus \{u_1,u_2\}$. 
Finally, this shows that $G$ satisfies the claim of the \bothV{L}{l}emma, and this contradiction concludes the proof of the \bothV{L}{l}emma for simple graphs.

Now given an arbitrary (multi\=/)graph $G$, let $\tilde{G}$ be the graph obtained from $G$ by deleting all loops and removing parallel edges. It is easy to see that $G$ is $K_{3,3}$-minor free, planar and $2$-vertex-connected if and only if the same is true for $\tilde{G}$. In addition, if $\tilde{G}$ has connected subgraphs $\tilde{G}_{i,j}$ satisfying Properties 1--3, by adding deleted loops and parallel edges again it is straightforward to obtain a decomposition $G_{i,j}$ of the multi-graph $G$ with the same properties. This concludes the proof of the \bothV{L}{l}emma.
\end{proof}

With these tools at hand, we are now ready for the proof of Theorem~\ref{K33free}. We need the notion of \emph{flow-critical graphs}, which was introduced \bothV{in}{by Silva and Lucchesi} \cite{flowcritical} as a possible approach towards Tutte's flow conjectures. Following the terminology of this paper, for $k \in \mathbb{N}$, we call a graph $G$ \emph{vertex-$k$-critical}, if $G$ is not $k$-flowable, but for every pair $u, v$ of distinct vertices, the{ subcontracted} graph $G/\{u,v\}$ (and hence every proper subcontraction of $G${ by Lemma~\ref{lemma:subcontflows}}) is $k$-flowable. Similarly, a graph is called \emph{edge-$k$-critical}, if $G$ is not $k$-flowable, but for every edge $e \in E(G)$, $G/e${ (and hence every graph obtained by contractions from $G$ by {Lemma~\ref{lemma:subcontflows}}}) is $k$-flowable. 
{See Fig.~\ref{fig:vertCrit} for an example illustrating these definitions.}

\begin{figure}[ht]
\centering
\includegraphics[scale=0.6]{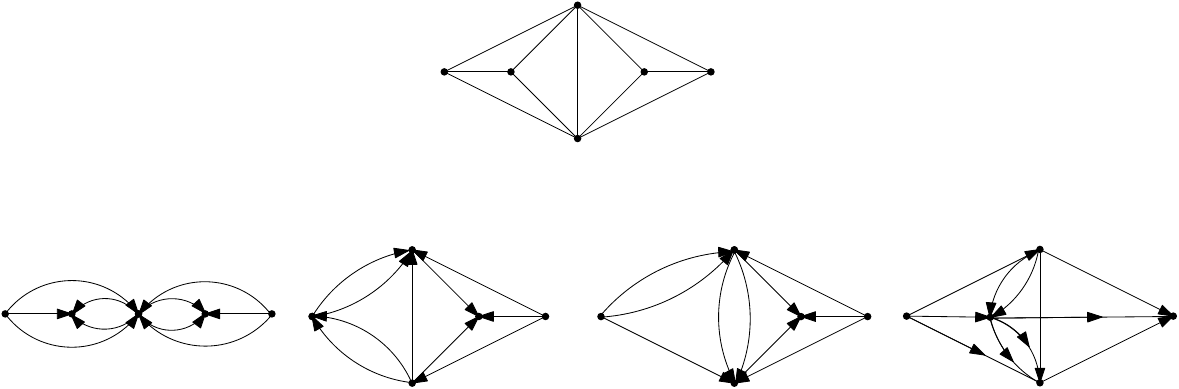}
\caption{Example of a vertex $3$-critical graph (top). It can be easily checked that this graph does not have a nowhere-zero $3$-flow, but all the $4$ non-isomorphic graphs obtained from it by identifying a pair of vertices admit modulo $3$-orientations (shown in the figures in the second row). Since the graph is vertex-$3$-critical, it is also edge-$3$-critical.}
\label{fig:vertCrit}
\end{figure}
Given a graph $G$ and a pair $u_1, u_2$ of vertices, let us say that an orientation $D$ of $G$ is a \emph{near-mod-$3$-orientation} of $G$ that \emph{misses $u_1$ and $u_2$}, if $d_D^+(x) \equiv d_D^-(x) \text{ (mod }3)$ for all vertices $x$ except $u_1$ and $u_2$. Since we have $\sum_{x \in V(G)}{(d_D^+(x) - d_D^-(x))}=0$, the excesses of $u_1$ and $u_2$ satisfy  $\text{exc}_D(u_1) \equiv \alpha \text{ (mod }3)$,  $\text{exc}_D(u_{\bothV{1}{2}}) \equiv -\alpha \text{ (mod }3)$ for some $\alpha \in \{-1,1\}$. Note that by reversing all edges of $D$, we can flip between the values $\alpha$ and $-\alpha$ for the parities of the excesses at $u_1$ and~$u_2$.

\restatekthreethreefree*
\begin{proof}
{Since we already know by Theorem~\ref{thm:3flowsuff} that every $3$-flowable graph is universally $3$-colorable, the task reduces to showing that every $K_{3,3}$-minor free graph which is \bothV{facially}{universally cell} $3$-colorable admits a nowhere zero $3$-flow.}
Suppose \bothV{the}{the latter} statement is false and let $G$ be a counterexample {with $n$ vertices and $m$ edges} minimizing {$n+m$}. This means that $G$ is {$K_{3,3}$-minor-free,} \bothV{facially}{universally cell} $3$-colorable but admits no nowhere-zero $3$-flow. We claim that $G$ has to be edge-$3$-critical: for every edge $e \in E(G)$, the graph $G/e$ is \bothV{facially}{universally cell} $3$-colorable by Proposition~\ref{subcontclosed} and, since it is a minor of $G$, it is $K_{3,3}$-minor free. From the minimality of $G$ we conclude that $G/e$ satisfies the claimed statement and hence admits a nowhere-zero $3$-flow.
Furthermore, $G$ has to be $2$-vertex-connected: clearly, $G$ is connected, \bothV{for}{as} otherwise one of the connected components of $G$ would form a smaller counterexample to the claim. Now suppose towards a contradiction that there was a cut-vertex $v \in V(G)$. Let $X_1,\ldots, X_k, k \ge 2$ be the connected components of $G-v$. For any $i$, the graph $G[X_i \cup \{v\}]$ is isomorphic to the graph obtained from $G$ by contracting all the edges in $G-X_i$. Since $G$ is edge-$3$-critical, \nbothV{this means that}{Lemma \ref{mod3or} implies that} there is a modulo-$3$-orientation $D_i$ of $G[X_i \cup \{v\}]$. Let $D$ be the orientation of $G$ obtained by joining the orientations $D_i, i=1,\ldots,k$. We now have $\text{exc}_{D}(x)=\text{exc}_{D_i}(x)$ for each $x \in X_i$ as well as $\text{exc}_{D}(v)=\sum_{i=1}^{k}{\text{exc}_{D_i}(v)}$. Since each~$D_i$ is a modulo-$3$-orientation, this means that $\text{exc}_{D}(x) \equiv 0 \text{ (mod }3)$ for all $x \in V(G)$. Thus,~$D$ is a modulo-$3$-orientation of $G$. This contradiction shows that $G$ must be $2$-vertex-connected.

Note that by Theorem~\ref{thm:flowcoloringduality}, $G$ is not planar. We can therefore apply Lemma~\ref{K5structure} to $G$ and obtain vertices $v_1,\ldots,v_5 \in V(G)$ and corresponding subgraphs $G_{i,j}, 1 \le i<j \le 5$ satisfying the properties 1--3 from the \bothV{L}{l}emma. Let us call a pair $1 \le i <j \le 5$ \emph{good} if $G_{i,j}$ admits a modulo-$3$-orientation and \emph{bad} if $G_{i,j}$ admits a near-mod-$3$-orientation that misses $v_i$ and $v_j$.

\paragraph{Claim 1.} Each pair $\{i,j\}$ is either good or bad, but not both.

\renewcommand{\qedsymbol}{\ensuremath{\blacktriangleleft}}
\begin{proof}[Proof{ of Claim 1}]
Let $1 \le i<j \le 5$ be given. Consider the graph {$F$} obtained from $G$ by contracting all edges in $E(G) \setminus E(G_{i,j})$. Since $G$ is edge-$3$-critical, \bothV{this graph}{$F$} admits a nowhere-zero $3$-flow. 
Since each $G_{i',j'}, \{i',j'\} \neq \{i,j\}$ is connected,{ in particular the contraction of $G_{i,k}$ and~$G_{k,j}$ for some $k\neq{i,j}$ contracts $v_i$ and $v_j$ into one vertex. Therefore} {$F$} is isomorphic to the graph $G_{i,j}/\{v_i,v_j\}$. Thus $G_{i,j}/\{v_i,v_j\}$ admits a modulo-$3$-orientation{ by Lemma~\ref{mod3or}}. Expanding such an orientation to $G_{i,j}$ by orienting possible edges between $v_i$ and $v_j$ arbitrarily, we obtain an orientation $D_{i,j}$ of $G_{i,j}$ such that $\text{exc}_{D_{i,j}}(x)$ is divisible by $3$ for all $x \in V(G) \setminus \{v_i,v_j\}$. Since the sum of all excesses is $0${ in any directed graph}, this means that \bothV{$\text{exc}_{D_{i,j}}(u_1), \text{exc}_{D_{i,j}}(u_2)$}{$\text{exc}_{D_{i,j}}(v_i), \text{exc}_{D_{i,j}}(v_j)$} are either both divisible by $3$, in which case $D_{i,j}$ is a modulo-$3$-orientation of $G_{i,j}$, or both are not divisible by $3$, which means that $D$ forms a near-mod-$3$-orientation of $G_{i,j}$ missing $v_i$ and $v_j$. This shows that $\{i,j\}$ is good or bad.

Let us now show that $\{i,j\}$ cannot be good and bad at the same time, and suppose towards a contradiction that \bothV{this was the case}{it is good and bad}. This means that for every $s \in \{-1,0,1\}$, there exists an orientation $D^s_{i,j}$ of $G_{i,j}$ such that $\text{exc}_{D^s_{i,j}}(x)$ is divisible by $3$ for all $x \in V(G_{i,j}) \setminus \{v_i,v_j\}$ and $\text{exc}_{D^s_{i,j}}(v_i) \equiv s, \text{exc}_{D^s_{i,j}}(v_j)\equiv -s \text{ (mod }3)$. Consider the graph $G/E(G_{i,j})$ which, by the edge-criticality of $G$, admits a modulo-$3$-orientation. Expanding this orientation to $G\setminus E(G_{i,j})$ we obtain an orientation $\overline{D}_{i,j}$ of $G-(V(G_{i,j}) \setminus \{v_i,v_j\})$ which has excesses divisible by $3$ at all vertices in $V(G) \setminus V(G_{i,j})$. Let $\alpha\coloneqq\text{exc}_{\overline{D}_{i,j}}(v_i) \text{ (mod }3)$. Since the excesses in $\overline{D}_{i,j}$ sum up to zero, we have $\text{exc}_{\overline{D}_{i,j}}(v_j) \equiv -\alpha \text{ (mod }3)$. We now claim that the orientation $D$ of $G$ obtained by joining $\overline{D}_{i,j}$ with the orientation $D^{-\alpha}_{i,j}$ of $G_{i,j}$ defines a modulo-$3$-orientation of $G$: since each vertex $x \in V(G) \setminus \{v_i,v_j\}$ is either only adjacent to edges in $\overline{D}_{i,j}$ or only to edges in $D^{-\alpha}_{i,j}$, its excess is divisible by $3$. On the other hand, by definition of the orientations $D^s_{i,j}, s=-1,0,1$, we have

$$\text{exc}_{D}(v_i)=\text{exc}_{\overline{D}_{i,j}}(v_i)+\text{exc}_{D^{-\alpha}_{i,j}}(v_i) \equiv \alpha-\alpha=0 \text{ (mod }3),$$
$$\text{exc}_{D}(v_j)=\text{exc}_{\overline{D}_{i,j}}(v_j)+\text{exc}_{D^{-\alpha}_{i,j}}(v_j) \equiv -\alpha+\alpha=0 \text{ (mod }3),$$
and hence, the excess at every vertex in the orientation $D$ is indeed divisible by $3$. This contradicts the fact that $G$ has no nowhere-zero flow and shows that our assumption was wrong. Finally, this implies that $\{i,j\}$ is good if and only if it is not bad, proving Claim 1.
\end{proof}

Let us now define an auxiliary simple graph $H$ on the vertex set $\{1,\ldots,5\}$ as follows: a~pair~$\{i,j\}$ with $1 \le i<j \le 5$ forms an edge in $H$ if and only if it is bad.

\paragraph{Claim 2.} The graph $H$ is vertex-$3$-critical.

\begin{proof}[Proof{ of Claim 2}]
Let us first verify that $H$ is not $3$-flowable: suppose towards a contradiction that there was a modulo-$3$-orientation $\vec{H}$ of $H$. For each directed edge $e=(i,j)$ in $\vec{H}$, by definition, there is a near-mod-$3$-orientation $\vec{G}_e$ of $G_{i,j}$ missing only $v_i$ and $v_j$. By reversing the orientation of all edges if required, we may assume that $\text{exc}_{\vec{G}_e}(v_i) \equiv 1 \text{ (mod }3)$ and \mbox{$\text{exc}_{\vec{G}_e}(v_j) \equiv -1 \text{ (mod }3)$}. Furthermore, for every good pair $e=\{i,j\} \in \binom{[5]}{2}\setminus E(H)$, let $\vec{G}_e$ be a modulo-$3$-orientation of $G_{i,j}$.

We now let $\vec{G}$ be the orientation of $G$ obtained by joining the orientations $\vec{G}_e$ for $e \in \binom{[5]}{2}$. We claim that $\vec{G}$ is a modulo-$3$-orientation of $G$.{ The idea is that the excess at vertex $v_k$ provided by $\vec{G}_e$ is the same as the excess provided by the edge $e$ to vertex $k$ in $\vec{H}$. Therefore summing up all excesses of the partial orientations $\vec{G}_e$ gives the same as summing up all excesses of $\vec{H}$, a modulo-3-orientation.}

For every $x \in V(G)$, we either have $x\in V(G_{i,j}) \setminus \{v_i,v_j\}$ for some $1 \le i<j \le 5$, and therefore $\text{exc}_{\vec{G}}(x)=\text{exc}_{\vec{G}_{i,j}}(x) \equiv 0 \text{ (mod }3)$, or $x=v_k$ for some $k \in [5]$ and therefore

$$\text{exc}_{\vec{G}}(\bothV{x}{v_k})=\sum_{e=(k,\ell) \in A(\vec{H})}{\text{exc}_{\vec{G}_e}(\bothV{x}{v_k})}+\sum_{e=(\ell,k) \in A(\vec{H})}{\text{exc}_{\vec{G}_{e}}(\bothV{x}{v_k})}+\sum_{e=\{k,\ell\} \in \binom{[5]}{2}\setminus E(H)}{\text{exc}_{\vec{G}_e}(\bothV{x}{v_k})}$$
$$\equiv \sum_{e=(k,\ell) \in A(\vec{H})}{1}+\sum_{e=(\ell,k) \in A(\vec{H})}{(-1)}+\sum_{e=\{k,\ell\} \in \binom{[5]}{2}\setminus E(H)}{0} $$
$$\equiv \text{exc}_{\vec{H}}(k) \equiv 0 \text{ (mod }3),$$ where we used that $\vec{H}$ is a modulo-$3$-orientation of $H$. Since $x$ was arbitrary, this means that $G$ is $3$-flowable. This contradiction shows that, indeed, $H$ does not admit a nowhere-zero $3$-flow.

Let now $v_i, v_j \in \{v_1,\ldots,v_5\}, i<j$ be given arbitrarily. To show that $H$ is vertex-$3$-critical, we must construct a modulo-$3$-orientation of the graph \bothV{$G/\{v_i,v_j\}$}{$H/\{i,j\}$}. For this, note that by the edge-criticality of $G$, the graph $G/E(G_{i,j})$ admits a modulo-$3$-orientation. This induces a modulo-$3$-orientation on each of the graphs $G_{k,\ell}$ for which $\{k,\ell\} \neq \{i,j\}$ is good, and a near mod-$3$-orientation missing $v_{k}$ and $v_{\ell}$ on each of the graphs $G_{k,\ell}$ for which $\{k,\ell\} \neq \{i,j\}$ is bad. For each pair $\{k,\ell\}$ let us denote by $\text{exc}_{k,\ell}(v_k), \text{exc}_{k,\ell}(v_\ell)$ the excess of the induced orientation of $G_{k,\ell}$ at $v_k$ respectively $v_\ell$. We clearly have $\text{exc}_{k,\ell}(v_k) \equiv -\text{exc}_{k,\ell}(v_\ell) \text{ (mod }3)$ for all $1 \le k<\ell \le 5$ and from the properties of the modulo-$3$-orientation on $G/E(G_{i,j})$ we deduce that for every $k \in \{1,\ldots,5\}\setminus \{i,j\}$, the following holds {(Note that we can exclude the good pairs in this sum since they contribute $0$ to it)}:

\begin{align}\label{eq1}
{\sum_{\substack{\ell \in [5] \setminus \{k\}\\ \{k,\ell\}\in E(H)}}{\hspace{-12pt}\text{exc}_{k,\ell}(v_k)} \equiv }\sum_{\ell \in [5] \setminus \{k\}}{\hspace{-12pt}\text{exc}_{k,\ell}(v_k)} \equiv 0 \text{ (mod }3),
\end{align}

and similarly:

\begin{align}\label{eq2}
{\sum_{\substack{\ell \in [5] \setminus \{i,j\}\\ \{i,\ell\}\in E(H)}}{\hspace{-12pt}\text{exc}_{i,\ell}(v_i)} + \sum_{\substack{\ell \in [5] \setminus \{i,j\}\\ \{j,\ell\}\in E(H)}}{\hspace{-12pt}\text{exc}_{j,\ell}(v_j)} \equiv }\sum_{\ell \in [5]\setminus \{i,j\}}{\hspace{-12pt}(\text{exc}_{i,\ell}(v_i)+\text{exc}_{j,\ell}(v_j))} \equiv 0 \text{ (mod }3).
\end{align}

Let us now define a partial orientation of $H$ as follows: orient an edge $\{k,\ell\} \in E(H)$ with $\{k, \ell\} \neq \{i,j\}$ from $k$ to $\ell$ if and only if $\text{exc}_{k,\ell}(v_k) \equiv 1 \text{ (mod }3)$. Recall that the edges of $H$ correspond to bad pairs $\{k,\ell\}$ only. In \bothV{the}{a} natural way, this partial orientation of $H$ induces a full {modulo-3-}orientation of the multi-graph $H/\{i,j\}$. Th\bothV{e}{is} claim \oldV{now} follows from observing that equations~\eqref{eq1} and~\eqref{eq2} above encode exactly the fact that the excess of each vertex in this orientation of $H/\{i,j\}$ is divisible by $3$. Hence, $H/\{i,j\}$ is $3$-flowable for all $i \neq j \in \{1,\ldots,5\}$, and this conclude{s} the proof of Claim 2.
\end{proof}

\renewcommand{\qedsymbol}{\openbox}

Contrary to the statement of Claim 2, no graph $G_5$ on exactly $5$ vertices is vertex-$3$-critical: we give a case distinction on the minimum degree:
\begin{itemize}
\item Minimum degree 0: The graph $G_5$ has an isolated vertex which we can identify into another one. This essentially deletes the isolated vertex, touching nothing else, yielding a graph which is 3-flowable if and only if $G_5$ is.
\item Minimum degree 1: Identifying two vertices other than the degree 1 vertex yields a graph with \bothV{minimum degree 1}{a degree 1 vertex}, which is thus not \bothV{bridgeless}{universally 3 cell-colorable by Proposition \ref{prop:leaf}} and therefore not 3-flowable{ by Theorem \ref{thm:3flowsuff}}.
\item Minimum degree 2: Identifying a degree 2 vertex $w$ into one of its neighbours $v$ yields a graph $G_4$. If $G_4$ is 3-flowable, we can extend the flow to $G_5$ by replacing the edge from $v$ that was an edge of $w$ in $G_5$ by a directed path of length 2 in the same direction, where $w$ is the middle vertex. 
\item Minimum degree 3 or 4: This means that{ $\overline{G_5}$,} the complement of $G_5${,} has maximum degree~1 or 0, so it has \bothV{0, 1 or 2}{at most two} non-adjacent edges. All three possibilities for such a graph are given in {Fig.}~\ref{fig:mindeg3flow}. As you can see, all of them are 3-flowable.
\end{itemize}

\begin{figure}[ht]
\centering
\includegraphics[scale=1]{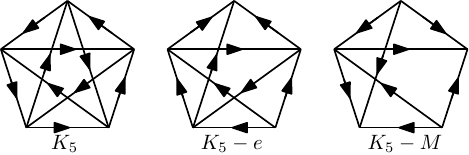}
\caption{Modulo-$3$-orientations for $K_{5}$ with at most 2 independent edges missing.}\label{fig:mindeg3flow}
\end{figure}

This contradiction shows that our initial assumption was wrong and a smallest counterexample $G$ to the claim cannot exist. This concludes the proof of Theorem~\ref{K33free}.
\end{proof}

Our next and final theorem of this section sums up some properties of smallest counterexamples to Conjecture~\ref{main}. We need the following definition: a graph $G$ is called \emph{$\mathbb{Z}_3$-connected}, if for every assignment $p\colon V(G) \rightarrow \mathbb{Z}_3$ with $\sum_{v \in V(G)}{p(v)}=0$ there exists an orientation $D$ of~$G$ such that $\text{exc}_{D}(v) \equiv p(v) \text{ (mod }3)$ for every $v \in V(D)$. Clearly, every $\mathbb{Z}_3$-connected graph is $3$-flowable, but the reverse is not true in general. Small examples of $\mathbb{Z}_3$-connected graphs are~$K_1$ with a loop, $K_2$ with two or more parallel edges, and $K_5$. We need the following auxiliary statements from the literature.

\begin{lemma}[cf. \cite{fan}, Proposition 1.2]\label{contractiblesubgraphs}
Let $G$ be a graph, and let $H \subseteq G$ be a subgraph. If $H$ is $\mathbb{Z}_3$-connected, then $G$ is $3$-flowable if and only if $G/V(H)$ is $3$-flowable.
\end{lemma}

\begin{lemma}[cf. \cite{flows}, Lemma 4.1.3]\label{smalledgecuts}
Let $G$ be a bridgeless graph. Assume that $G$ is not $3$-flowable and has an (inclusion-wise) minimal edge\noldV-cut $S$ of size at most $3$. Let $X_1, X_2$ be the components of $G-S$. Then either $G/X_1$ or $G/X_2$ is not $3$-flowable.
\end{lemma}

\begin{theorem}\label{thm:mincex}
Let $G$ be a counterexample to Conjecture~\ref{main} such that the claim of Conjecture~\ref{main} is satisfied for all graphs $G'$ with $|V(G')|<|V(G)|$ or $|V(G')|=|V(G)|$ and $|E(G')|<|E(G)|$. Then
\begin{itemize}
\item $G$ is vertex-$3$-critical,
\item $G$ contains no $\mathbb{Z}_3$-connected subgraph except $K_1$ and (thus) is a simple graph,
\item $G$ is $3$-edge-connected and every $3$-edge-cut consists of the edges incident to a cubic vertex,
\item $G$ has a vertex of degree at least $4$ and \oldV{has }a $K_{3,3}$-minor.
\end{itemize}
\end{theorem}

\begin{proof}
Since $G$ is a counterexample to Conjecture~\ref{main}, $G$ has no \bothV{$K_{3,n}^+$}{$K_{3,n-3}^+$} with $n \ge 7$ as a subcontraction, is \bothV{facially}{universally cell} $3$-colorable (and hence bridgeless) but not $3$-flowable. For every distinct $u, v \in V(G)$, by Proposition~\ref{subcontclosed} also $G/\{u,v\}$ is \bothV{facially}{universally cell} $3$-colorable. Since $G$ has no subcontraction isomorphic to a \bothV{$K_{3,n}^+, n \ge 4$}{$K_{3,n-3}^+, n \ge 7$}, the same must be true for the subcontraction $G/\{u,v\}$ of $G$. Since $|V(G/\{u,v\})|<|V(G)|$, $G/\{u,v\}$ is $3$-flowable. This proves that~$G$ is vertex-$3$-critical. 
From Lemma~\ref{contractiblesubgraphs} and the vertex-criticality of $G$ we deduce that $G$ contains no $\mathbb{Z}_3$-connected subgraphs of order at least two.
 This fact now rules out parallel edges since $K_2$ with two or more parallel edges \bothV{are}{is} $\mathbb{Z}_3$-connected. {Furthermore, $G$ has no loops by edge\noldV-minimality, as they can be added to a modulo-$3$-orientation by orienting them arbitrarily since they do not impact the excess of their vertex.} Thus $G$ is simple.
 From Lemma~\ref{smalledgecuts} \oldV{and the vertex-criticality of $G$,} it follows that if $T$ is a\bothV{n}{ minimal} edge\noldV-cut in $G$ of size at most $3$ separating the parts $X_1$ and $X_2$ of $V(G)$, then {$G/X_1$ or $G/X_2$ is not $3$-flowable. By vertex-criticality all proper subcontractions of $G$ are $3$-flowable, thus one of $X_1$ or $X_2$ must be trivial, that is} $\min\{|X_1|,|X_2|\}=1$. Hence, in this case there is a vertex $v \in V(G)$ such that $T$ consists of all edges incident to $v$. If $|T|=2$, then $v$ is a vertex of degree $2$. Then, however, contracting one of the edges incident with $v$ produces a graph smaller than $G$ which still has no $3$-flow, contradicting the minimality assumption on $G$. Hence, we must have $|T| \ge 3$ for every edge-cut $T$, showing that $G$ is $3$-edge-connected. It also follows that every $3$-edge-cut consists of the edges incident to a cubic vertex, hence we have proved the third item. The last item follows directly from Theorem~\ref{scubic} and Theorem~\ref{K33free}.
\end{proof}

\section{Conclusive Remarks}\label{sec:schluss}
Apart from the obvious challenges to decide Conjecture~\ref{main} and to obtain a better understanding of the class of \bothV{facially}{universally cell} $3$-colorable graphs, we have an interesting open question towards the computational complexity of recognizing \bothV{facially}{universally cell} $3$-colorable graphs. \bothV{For planar graphs, \NP-completeness can be deduced as follows.}{Recall that in Corollary~\ref{planarhard} we have established the \NP-completeness of recognising whether a given \emph{planar} graph is universally cell $3$-colorable.}

While this clearly suggests hardness for general graphs as well, containment in \NP\ remains unclear, since \bothV{facial}{universal cell} $3$-colorability cannot be verified via $3$-flowability in this case. 
%We can generalize our result slightly to graphs of constant crossing number.
%
%\begin{corollary}
%Let $c\in \N$. Deciding whether a given graph of crossing number at most $c$ is \bothV{facially}{universally cell} $3$-colorable is \NP -complete.
%\end{corollary}
%
%
%\begin{proof}
%For planar graphs
%\end{proof}
%
%The general question remains unsolved though.

\begin{question}
Is deciding whether an input graph is \bothV{facially}{universally cell} $3$-colorable contained in~\NP?
\end{question}

\paragraph{Acknowledgements} We are indebted to Andrew Newman, G\"{u}nter Rote, and L\'{a}szl\'{o} Kozma, who have been involved in fruitful discussions during the Problem Solving Workshop of the Research Training Group `Facets of Complexity' at Kloster Chorin in December 2019.

\appendix
\section{Appendix: Proof of Proposition~\ref{prop:topssuffice}}
\restatetopssuffice*
\begin{proof}
Suppose towards a contradiction that all good drawings of $G$ are \bothV{face-}{cell }$3$-colorable, but that there exists a drawing $\Gamma$ of $G$ which is not \bothV{face-}{cell }$3$-colorable. Further suppose $\Gamma$ is chosen such that among all non-\bothV{face-}{cell }$3$-colorable drawings of \bothV{$\Gamma$}{$G$}, it minimizes the number of triples $(e_1,e_2,p)$, where $e_1, e_2 \in E(G)$ are distinct edges and $p \in \gamma^\circ(e_1) \cap \gamma^\circ(e_2)$, or $e_1=e_2$ and $p$ is a self-intersection of the edge $e_1$. In the rest of the proof, we call such triples \emph{crossing triples}. 

For the case that $\Gamma$ contains three or more edges intersecting in a common point $p$, we can consider $\varepsilon>0$ small enough such that $B_\varepsilon(p)$ contains no other vertices or intersection-points of~$\Gamma$. We can then redraw the edges within the ball $B_\varepsilon(p)$ such that they pairwise intersect once within the ball, but at every such intersection, we have only two intersecting edges. A simple way of achieving this is by glueing a drawing of a simple pseudoline-arrangement into the ball~$B_\varepsilon(p)$ such that every pseudoline connects a pair of points on the boundary of $B_\varepsilon(p)$ belonging to the same edge of $\Gamma$. This local redrawing-process is illustrated in {Fig.}~\ref{highdegree}.

\begin{figure}[ht]
\centering
\includegraphics[scale=1]{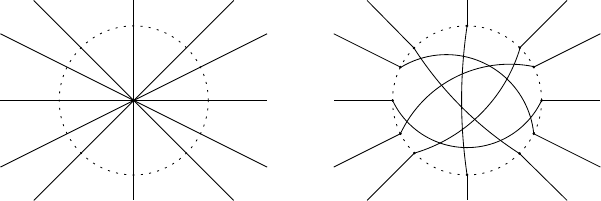}
\caption{Locally replacing a common intersection of more than two edges by a simple pseudoline-arrangement.}\label{highdegree}
\end{figure}

Note that this process changes the dual graph $G^\top(\Gamma)$ only in a way that new vertices and edges are added, but no connections between originally adjacent \bothV{face}{cell}s are being lost. Hence, the chromatic number of $G^\top(\Gamma)$ cannot decrease by this process, and hence $\Gamma$ is still not \bothV{face-}{cell }$3$-colorable afterwards.
Further note that the process leaves the number of crossing triples unaffected, and hence the minimality assumption on $\Gamma$ remains valid. Hence, possibly after performing this operation at every intersection of more than two edges in $\Gamma$, we may assume from now on that at most two edges in $\Gamma$ intersect in a common point.

Since $\Gamma$ is not \bothV{face-}{cell }$3$-colorable, $\Gamma$ is not a good drawing. By our definition of a good drawing, this means that at least one of the following cases must occur.

\begin{figure}[ht]
\centering
\includegraphics[scale=0.7]{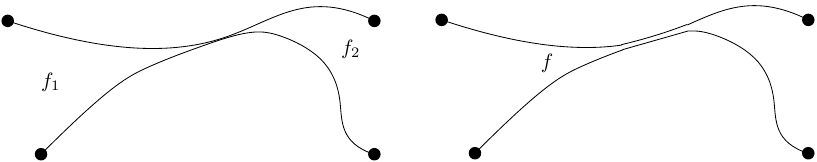}
\caption{Removing a touching.}\label{touching}
\end{figure}

\paragraph{Case 1.} There is a pair $e_1=u_1v_1, e_2=u_2v_2$ of   (possibly adjacent or equal) edges that touch in $p$. By locally rerouting $\gamma(e_1)$ and $\gamma(e_2)$ in a small neighborhood around $p$, we can create the curves $\gamma^\ast(e_i), i=1,2$ from $u_i$ to $v_i$ by avoiding the touching at $p$ and leaving a small gap between the curves, see {Fig.}~\ref{touching} for an illustration. Replacing $\gamma(e_i), i=1,2$ by $\gamma^\ast(e_i), i=1,2$ yields a new drawing 
$\Gamma^\ast$ of $G$, in which $(e_1,e_2,p)$ is no crossing triple any more, and no new intersections between edges have been created. Hence, the number of crossing triples in $\Gamma^\ast$ is smaller than in~$\Gamma$. By the minimality assumption on $\Gamma$, this means that $\Gamma^\ast$ admits a \bothV{face-}{cell }$3$-coloring. Let~$f_1$ and~$f_2$ be the two \bothV{faces}{cells} of $\Gamma$ incident to $p$ which are merged into a common greater \bothV{face}{cell} $f$ when moving from $\Gamma$ to $\Gamma^\ast$. It is now clear that given any \bothV{face-}{cell }$3$-coloring of $\Gamma^\ast$, by assigning to both $f_1, f_2$ the color of $f$ in the proper coloring of $\Gamma^\ast$ yields a proper \bothV{face-}{cell }$3$-coloring of $\Gamma$. This contradicts our initial assumption on $\Gamma$, so $\Gamma$ cannot contain any touchings.

\begin{figure}[ht]
\centering
\includegraphics[scale=0.7]{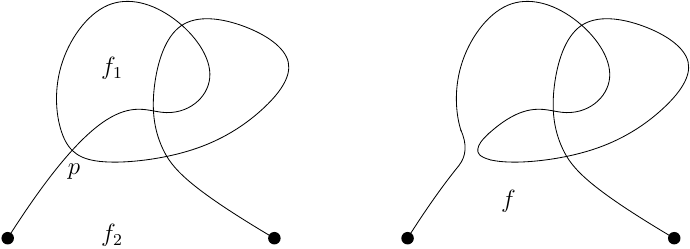}
\caption{Removing a self-intersection of an edge.}\label{selfsec}
\end{figure}

\paragraph{Case 2.} There is an edge $e \in E(G)$ which is either a non-loop edge such that $\gamma:=\gamma(e)$ self-intersects or a loop-edge that self-intersects in its interior. In this case, let $p$ be such a point of self-intersection. Let $\delta \subseteq \gamma$ be a closed curve starting and ending at $p$. As illustrated in {Fig.}~\ref{selfsec}, we can now reroute $\gamma$ such that it now traverses the loop $\delta$ in opposite direction, naming it~$\gamma^\ast$. We do not create new intersections between edges by modifying $\gamma$ into $\gamma^\ast$, nor delete any intersections, and hence in the drawing $\Gamma^\ast$ obtained from $\Gamma$ by replacing $\gamma$ with $\gamma^\ast$, the number of crossing triples is the same. $\Gamma^\ast$ contains a touching though and still fulfills the minimality conditions of $\Gamma$, which is impossible by Case 1, ruling out Case 2.

%admits a face-$3$-coloring. Moving from $\Gamma$ to $\Gamma^\ast$, two faces, say $f_1$ and $f_2$, are merged into a unified face $f$ in $\Gamma^\ast$, while the adjacencies to and between the remaining faces stay unchanged. It is now clear that given any face-$3$-coloring of $\Gamma^\ast$, by assigning to both $f_1, f_2$ the color of $f$ in the proper coloring of $\Gamma^\ast$ yields a proper face-$3$-coloring of $\Gamma$. This contradicts our initial assumption on $\Gamma$ and rules out Case 2.

\begin{figure}[ht]
\centering
\includegraphics[scale=0.7]{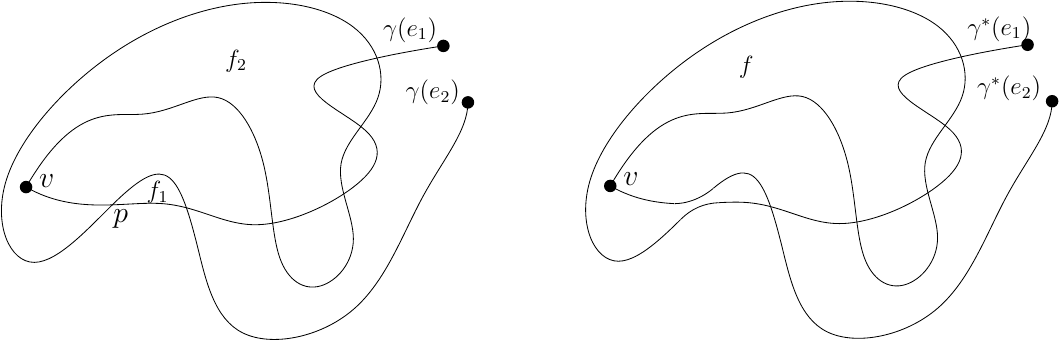}
\caption{Removing an intersection of two adjacent edges.}\label{adjacentedges}
\end{figure}

\paragraph{Case 3.} There is a pair $e_1, e_2$ of distinct edges with a common endpoint $v$ such that we have \mbox{$\gamma^\circ(e_1) \cap \gamma^\circ(e_2)=\{p\}$}. Then \emph{rerouting} at $p$, that is, exchanging the pieces of the curves $\gamma(e_i)$,~\mbox{$i=1,2$}, between $v$ and $p$, yields a pair of new curves~$\gamma^\ast(e_1), \gamma^\ast(e_2)$. See {Fig.}~\ref{adjacentedges} for an illustration. Replacing $\gamma(e_i), i=1,2$ by $\gamma^\ast(e_i), i=1,2$ yields a new drawing $\Gamma^\ast$ of $G$, in which~$(e_1,e_2,p)$ is a touching, but the number of crossing triples in $\Gamma^\ast$ is the same as in $\Gamma$. So the minimality assumption of $\Gamma$ holds for $\Gamma^\ast$ as well and by Case 1, this is impossible and rules out Case 3.

\begin{figure}[ht]
\centering
\includegraphics[scale=0.7]{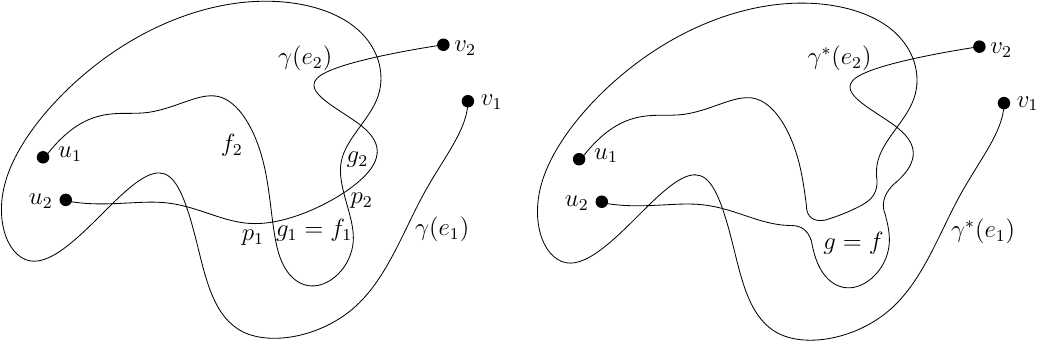}
\caption{Removing intersections of two non-adjacent edges.}\label{nonadjacentedges}
\end{figure} 

\paragraph{Case 4.} There is a pair $e_1=u_1v_1, e_2=u_2v_2$ of distinct non-adjacent edges such that \mbox{$\gamma^\circ(e_1) \cap \gamma^\circ(e_2)$} contains at least two distinct points. Let $p_1 \neq p_2 \in \gamma^\circ(e_1) \cap \gamma^\circ(e_2)$ be two such points. Then \emph{rerouting} at $p_1$ and $p_2$, that is, exchanging the segments of the curves $\gamma(e_i), i=1,2$ between $p_1$ and $p_2$, yields a pair of new curves $\gamma^\ast(e_i), i=1,2$ routed between $u_i$ and $v_i$. See \bothV{Figure}{Fig.}~\ref{nonadjacentedges} for an illustration. Replacing $\gamma(e_i), i=1,2$ by $\gamma^\ast(e_i), i=1,2$ yields a new drawing $\Gamma^\ast$ of~$G$, in which $(e_1,e_2,p_1)$ and $(e_1,e_2,p_2)$ are touchings, but the number of crossing triples in $\Gamma^\ast$ is the same as in $\Gamma$. So the minimality assumption of $\Gamma$ holds for $\Gamma^\ast$ as well and by Case 1, this is a contradiction, so it rules out Case 4.

\paragraph{Case 5.} There is a loop $e \in E(G)$ incident to a vertex $v \in V(G)$ such that $\gamma(e)$ intersects other edges in the drawing $\Gamma$. In this case, we can define a new drawing $\Gamma^\ast$ of $G$ which is obtained from~$\Gamma$ by first removing $\gamma(e)$ from the drawing and then redrawing the loop $e$ within a \bothV{face}{cell} of $\Gamma-\gamma(e)$ incident with $v$, such that it does not intersect any other feature of the drawing. See \bothV{Figure}{Fig.}~\ref{loop} for an illustration. Clearly, the number of crossing triples in $\Gamma^\ast$ is strictly smaller than in $\Gamma$. By the minimality assumption on $\Gamma$, this means that $\Gamma^\ast$ admits a \bothV{face-}{cell }$3$-coloring. By redrawing the crossing-free loop $e$ in $G_\text{isc}(\Gamma^\ast)$ such that it takes the same position as in $\Gamma$, we obtain a drawing $\Gamma'$ of $G_\text{isc}(\Gamma^\ast)$ whose induced cell-decomposition is the same as the one induced by $\Gamma$. It follows now from Corollary~\ref{cor:crossingsmakeiteasier} that with $\Gamma^\ast$ also $\Gamma$ must have a \bothV{face-}{cell }$3$-coloring. This contradicts our initial assumption on $\Gamma$ and rules out Case 5.

\begin{figure}[ht]
\centering
\includegraphics[scale=0.7]{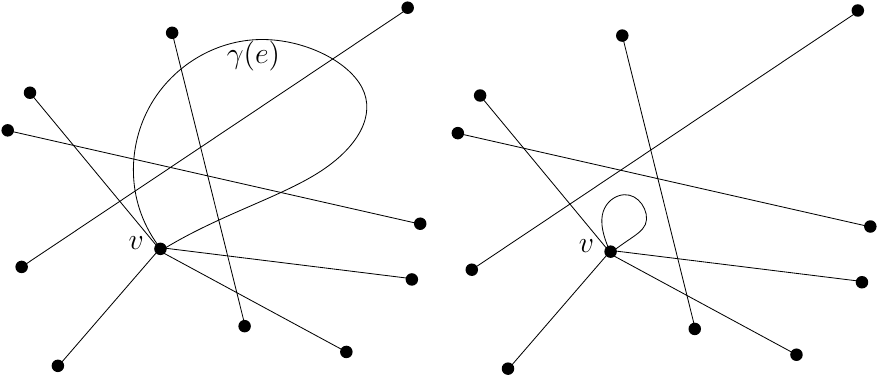}
\caption{Redrawing a loop to avoid crossings.}\label{loop}
\end{figure}

Since we have arrived at a contradiction in each case, we conclude that our initial assumption, namely that there is a drawing $\Gamma$ of $G$ which is not \bothV{face-}{cell }$3$-colorable, was wrong. Hence, $G$ is \bothV{facially}{universally cell} $3$-colorable, and this concludes the proof of the proposition.
\end{proof}

\end{document}